\documentclass[a4paper,12pt]{article}
\usepackage{amsmath,amssymb} \usepackage{euler,eucal}
\usepackage{ifxetex}
\ifxetex
  \usepackage{fontspec} \usepackage{polyglossia}
  \defaultfontfeatures{Mapping=tex-text}
  \setromanfont{URW Palladio L}
\else
  \usepackage[utf8]{inputenc}
  \usepackage[T1]{fontenc}
  \usepackage[english]{babel}
  \usepackage{tgpagella}
\fi
\usepackage[unicode=true,plainpages=false]{hyperref}
\usepackage{graphicx}
\hypersetup{colorlinks=false,pdfborder={0 0 0.1},linkcolor=magenta,anchorcolor=magenta,urlcolor=blue,citecolor=blue,
          pdftitle={Quasioptimality of maximum-volume cross interpolation of tensors},
          pdfauthor={D. V. Savostyanov},
          pdfkeywords={high-dimensional problems, tensor train format, maximum-volume principle, cross interpolation},
          }
\newcommand*{\doi}[1]{doi: \href{http://dx.doi.org/#1}{#1}}
\usepackage[left=25mm,right=25mm,top=25mm,bottom=30mm,a4paper]{geometry}
\usepackage{algorithmic}
\usepackage[ruled]{algorithm}
\usepackage{amsthm}
\theoremstyle{definition}

\newtheorem{remark}{Remark}

\newtheorem{lemma}{Lemma}

\newtheorem{theorem}{Theorem}

\usepackage{tikz} \usetikzlibrary{positioning,shapes}
\let\eps=\varepsilon

\let\leq=\leqslant
\let\geq=\geqslant

\def\I{\mathcal{I}} \def\J{\mathcal{J}} 
\def\O{\mathcal{O}}
\def\E{\mathcal{E}}
\def\s{\mathbf{s}} \def\t{\mathbf{t}} 
\def\i{i_1,\ldots,i_d}
\def\rank{\mathop{\mathrm{rank}}\nolimits}
\def\det{\mathop{\mathrm{det}}\nolimits}
\def\vol{\mathop{\mathrm{vol}}\nolimits}
\def\maxvol{\mathop{\mathrm{maxvol}}\nolimits}
\def\lft{\triangleleft}
\def\rgt{\triangleright}
\def\trans{*}
\def\new{\star}
\begin{document}
\author{Dmitry V. Savostyanov%
\thanks{University of Southampton, Department of Chemistry, Highfield Campus, Southampton SO17 1BJ, United Kingdom 
        ({\tt dmitry.savostyanov@gmail.com})}
\thanks{Institute of Numerical Mathematics of Russian Academy of Sciences, Gubkina 8, Moscow 119333, Russia}
}
\title{Quasioptimality of maximum--volume cross interpolation of tensors%
\thanks{Partially supported by
         RFBR grants 11-01-00549-a, 12-01-33013, 12-01-00546-a, 12-01-91333-nnio-a,
         Rus. Fed. Gov. project 16.740.12.0727 
         at INM RAS
         and EPSRC grant EP/H003789/1 at the University of Southampton.
         }}
\date{June 15, 2013}
\maketitle

\begin{abstract}
We consider a cross interpolation of high--dimensional arrays in the tensor train format.
We prove that the maximum--volume choice of the interpolation sets provides the quasioptimal interpolation accuracy, 
that differs from the best possible accuracy by the factor which does not grow exponentially with dimension.
For nested interpolation sets we prove the interpolation property and propose greedy cross interpolation algorithms.
We justify the theoretical results and test the speed and accuracy of the proposed algorithm with convincing numerical experiments. 
\par {\it AMS:}
15A69,  
15A23,  
65D05,  
65F99.  
\end{abstract}

\section{Introduction}
As demand for big data analysis grows, high--dimensional data and algorithms have become increasingly important in scientific computing.
The total number of entries in a \emph{tensor} (an array with $d$ indices) grows exponentially with dimension $d.$
Even for a moderate $d,$ it is impossible to process, store or compute all elements of a tensor by standard methods. 
This issue is known in numerical analysis and related areas as the \emph{curse of dimensionality}.
Different techniques are used to relax or to overcome this problem, e.g. low--parametrical representation on {sparse grids}~\cite{smolyak-1963,griebel-sparsegrids-2004}, (Markov chain) Monte Carlo sampling in statistics~\cite{hastings-mcmc-1970}, model/dimensionality reduction, etc.
Significant progress has been made in the development and understanding of the \emph{tensor product} methods (see reviews~\cite{kolda-review-2009,khor-survey-2011,hackbusch-2012,larskres-survey-2013}).

The tensor product methods implement the \emph{separation of variables} at the discrete level, which in the two--dimensional case is known as the \emph{low rank decomposition} of a matrix.
Several approaches have been developed to generalize rank--structured low--parametrical models to tensors (see~\cite{kolda-review-2009} for details), and a particularly simple and efficient~\emph{tensor train} (TT) format has been proposed recently~\cite{osel-tt-2011}.
It is equivalent to the \emph{matrix product states} (MPS) introduced in the quantum physics community to represent the quantum states of the many--body systems~\cite{fannes-mps-1992,klumper-mps-1993}.
The optimization algorithms for the MPS include \emph{alternating least squares} (ALS) algorithm, which works with the fixed tensor structure, and the \emph{density matrix renormalization group} (DMRG) algorithm~\cite{white-dmrg-1993,ostlund-dmrg-1995}, which adaptively changes the \emph{ranks} of the tensor format, and manifests much faster convergence in numerical experiments.
When the TT format was re-discovered in the numerical linear algebra community, both the ALS and DMRG schemes were adapted for other high--dimensional problems and novel algorithms were proposed. 
As a result we can use the TT/MPS format to approximate high--dimensional data and perform algebraic operations~\cite{schollwock-dmrg-mps-2010} (cf. \cite{osel-tt-2011,Os-mvk2-2011}), solve linear systems~\cite{jeckelmann-dmrgsolve-2002,holtz-ALS-DMRG-2012,DoOs-dmrg-solve-2011,ds-amr1-2013,ds-amr2-2013}, compute the multidimensional Fourier transform~\cite{dks-ttfft-2012} and discrete convolution~\cite{khkaz-conv-2011}.
With these algorithms in hand, high--dimensional scientific computations become possible as soon as all data are somehow translated into the TT format.

It is crucial, therefore, to develop algorithms which construct the approximation of a given high--dimensional array in the tensor format.
For some function--related tensors, the TT representation is written explicitly (see e.g.~\cite{khor-qtt-2011,osel-constr-2013}).
In general, although every entry of a tensor can be computed \emph{on demand} (by a formula or as a solution of a feasible problem, e.g. PDE in three dimensions), all elements cannot be computed in a reasonable time. 
The question arises naturally whether a tensor can be \emph{reconstructed} or \emph{interpolated} in the TT format from a few elements, also known as \emph{samples}.

For matrices, i.e. $2$--tensors, this question is well studied. 
We know that a rank--$r$ matrix is recovered from a \emph{cross} of $r$ rows and columns if the submatrix on their intersection is nonsingular.
When data are not exactly represented by the low--rank model, the accuracy of the \emph{cross interpolation} depends crucially on the chosen cross.
A notable choice is the~\emph{maximum volume} cross, which has the $r\times r$ submatrix with the maximum determinant in modulus on the intersection.
For this cross, the interpolation accuracy differs from the accuracy of the best possible approximation by the factor $\O(r^2),$ i.e. is \emph{quasioptimal}~\cite{schneider-cross2d-2010,gt-skel-2011}.

For tensors in the TT format an analog of the cross interpolation formula is given in~\cite{ot-ttcross-2010}. 
It reconstructs a tensor from a few samples under mild non-singularity conditions, if the TT representation is exact.
For the approximate case, the ALS type algorithm is suggested in~\cite{ot-ttcross-2010}, which searches for the better crosses in order to improve the approximation accuracy.
The rank--adaptive DMRG--like version of this algorithm is proposed in~\cite{so-dmrgi-2011proc}.
These algorithms are~\emph{heuristic}, as well as interpolation algorithms developed for other tensor formats, e.g. the Tucker~\cite{ost-tucker-2008,ost-chem-2010} and the \emph{hierarchical Tucker} (HT) format~\cite{lars-htcross-2013,bebe-maca-2013pre}.

The accuracy of the cross interpolation of tensors has not been well studied yet.
For the $3$--dimensional Tucker model the quasioptimality with the factor $\O(r^3)$ is shown in~\cite{ost-tucker-2008}.
In $d$ dimensions we can expect an excessively large coefficient $\O(r^d),$  cf. $\O(r^{2d})$  for the HT format~\cite{lars-htcross-2013}.
The main result of this paper is more optimistic. 
The quasioptimality of the maximum volume cross interpolation is generalized to the TT format with the coefficient $(2r+\kappa r+1)^{\lceil \log_2 d \rceil+2}$ that \emph{does not necessarily grow exponentially} with $d.$ 

The paper is organized as follows.
Sec.~\ref{sec:def} presents notation and definitions.
In Sec.~\ref{sec:mvol} the quasioptimality of the maximum--volume cross interpolation is proven.
In Sec.~\ref{sec:emb} the interpolation on \emph{nested} sets is considered, which reduces the search space, but results in the larger quasioptimality constant.
In Sec.~\ref{sec:emb2} the interpolation property for the nested sets is shown.
In Sec.~\ref{sec:alg} practical cross interpolation algorithms for matrices are recalled and similar algorithms for tensor trains are proposed.
In Sec.~\ref{sec:num} the coefficient of the quasioptimality is measured for randomly generated tensors, and speed and accuracy of the proposed algorithm is demonstrated with numerical experiments.

\section{Notation, definitions and preliminaries} \label{sec:def}
The \emph{tensor train} (TT) decomposition of a tensor $A=\left[A(i_1,\ldots,i_d)\right]$ is written as follows 
\begin{equation}\label{eq:tt}
 \begin{split}
 A(i_1,\ldots,i_d) & = \sum_{\s} X^{(1)}(i_1,s_1) X^{(2)}(s_1,i_2,s_2) \ldots X^{(d-1)}(s_{d-2},i_{d-1},s_{d-1}) X^{(d)}(s_{d-1},i_d)
   \\ & = \sum_{\s} \prod_{k=1}^d X^{(k)}(s_{k-1},i_k,s_k).
 \end{split}  
\end{equation}
In this equation $i_k=1,\ldots,n_k,$ $k=1,\ldots,d,$ are \emph{mode} or physical indices, and $s_k=1,\ldots, r_k$ are auxiliary \emph{rank} indices. 
Values $n_k$ are referred to as~\emph{mode sizes} of a tensor, and $r_k$ are~\emph{tensor train ranks} or TT--ranks.
Summation over $\s=(s_1,\ldots,s_{d-1})$ means summation over all pairs of auxiliary indices $s_1,\ldots,s_{d-1},$ where each index runs through all possible values.
We use elementwise notation, i.e. assume that all equations hold for all possible values of \emph{free} indices.
Therefore, Eq.~\eqref{eq:tt} represents every entry of a tensor by the product of matrices, where each $X^{(k)}(i_k)=[X^{(k)}_{s_{k-1},s_k}(i_k)]$ has size $r_{k-1}\times r_k$ and depends on the~\emph{parameter} $i_k.$
The three--dimensional array $X^{(k)}=\left[X^{(k)}(s_{k-1},i_k,s_k)\right]$ is referred to as \emph{TT--core}.
To unify the notation, we introduce the virtual \emph{border ranks} $r_0=r_d=1$ and consider 
$[X^{(1)}(i_1,s_1)]=[X^{(1)}(s_0,i_1,s_1)]$ and $[X^{(d)}(s_{d-1},i_d)]=[X^{(d)}(s_{d-1},i_d,s_d)]$ as $3$--tensors.

The elementwise notation allows us to reshape tensors into vectors or matrices simply by moving indices. 
We have done this to present the TT--core $\left[X^{(k)}(s_{k-1},i_k,s_k)\right]$ as the parameter--dependent matrix $[X^{(k)}_{s_{k-1},s_k}(i_k)].$
More complicated transformations can be expressed by \emph{index grouping}, which combines indices $i_1,\ldots,i_d$ in the single multi--index $\overline{i_1\ldots i_d}.$%
\footnote{ 
The multi--index is usually defined by either the \emph{big--endian} convention
$\overline{i_1\ldots i_d}=i_d+(i_{d-1}-1)n_d +\ldots+(i_1-1)n_2\ldots n_d$
or the \emph{little--endian} convention
$\overline{i_1\ldots i_d}=i_1+(i_2-1)n_1 +\ldots+(i_d-1)n_1\ldots n_{d-1}.$
The big--endian notation is similar to numbers written in the positional system, while the {little--endian} notation is used in numerals in the Arabic scripts and is consistent with the \textsc{Fortran} style of indexing.
The exact formula which maps indices to the multi--index is not essential in this paper.  
} 
For example, the $k$--th \emph{unfolding} of a tensor is the $(n_1\ldots n_k) \times (n_{k+1}\ldots n_d)$ matrix with elements
\begin{equation}\nonumber 
A^{\{k\}}(i_{\leq k},i_{>k}) = A^{\{k\}}(\overline{i_1\ldots i_k}, \overline{i_{k+1}\ldots i_d}) =  A(i_1,\ldots,i_d).
\end{equation}
Here and further we use the following shortcuts to simplify the notation
$$
i_{\leq k} = \overline{i_1\ldots i_k}, \quad i_{>k} = \overline{i_{k+1}\ldots i_d}, 
\qquad\mbox{and}\qquad i_{b:c}=\overline{i_b\ldots i_c}. 
$$
For $A$ in the TT--format~\eqref{eq:tt} it holds  $\rank A^{\{k\}}=r_k.$
In~\cite{osel-tt-2011} the reverse is proven: for any tensor $A$ there exists the representation~\eqref{eq:tt} with TT--ranks $r_k=\rank A^{\{k\}}.$
This gives the term \emph{TT--rank} the definite algebraic meaning.

For a $m\times n$ matrix $A=\left[A(i,j)\right]$ the \emph{cross} (or \emph{skeleton}) interpolation is written as follows
\begin{equation}\label{eq:im}
 A(i,j) \approx \tilde A(i,j)  = \sum_{s,t} A(i,\J_t) \left[A(\I_s,\J_t)\right]^{-1} A(\I_s,j).
\end{equation}
Here sets $\I=\{\I_1,\ldots,\I_r\}$ and $\J=\{\J_1,\ldots,\J_r\}$ define the positions of the interpolation rows and columns, respectively.  
The summation over $s,t=1,\ldots,r$ ties the pairs of subsets together, similarly to the pairs of indices in~\eqref{eq:tt}.
In the matrix form the right hand side of~\eqref{eq:im} is the product of $m \times r$ matrix of columns, the inverse of $r\times r$ submatrix at the intersection and $r\times n$ matrix of rows. 
The essential property of the interpolation is that~\eqref{eq:im} is~\emph{exact} on its cross
\begin{equation}\label{eq:ip}
 A(i,j) = \tilde A(i,j) = \sum_{s,t} A(i,\J_t) \left[A(\I_s,\J_t)\right]^{-1} A(\I_s,j), \qquad 
    \mbox{if}\quad  i\in\I \mbox{ or } j\in\J.
\end{equation}

When $A$ is not exactly a rank-$r$ matrix, the choice of interpolation sets $\I,\J$ may affect the interpolation accuracy significantly.
A good choice of $A_\Box=\left[A(\I,\J)\right]$  is the~\emph{maximum--volume} $r\times r$ submatrix, such that  $\vol A_\Box = |\det A_\Box| $ is maximal over all possible choices of $\I$ and $\J.$
Assuming that the ranks (sizes of submatrices) are defined \emph{a priori}, we denote this choice by
\begin{equation}\nonumber
  \left[\I,\J\right] = \arg\max_{\I',\J'}\vol [A(\I',\J')], \qquad\mbox{or}\qquad  [\I,\J] = \maxvol A.
\end{equation}
For $\I,\J$ chosen by the maximum--volume principle, the following \emph{quasioptimality} statements are proven in~\cite{gt-maxvol-2001} and~\cite{schneider-cross2d-2010,gt-skel-2011}, respectively.
\begin{equation}\label{eq:qm}
 \begin{split}
   \| A - \tilde A \|_C & \leq (r+1)^{\phantom{2}}   \min\nolimits_{\rank X=r} \| A - X \|_2, \\
   \| A - \tilde A \|_C & \leq (r+1)^2               \min\nolimits_{\rank X=r} \| A - X \|_C. 
  \end{split}
\end{equation}
Another important property of the maximum--volume submatrix is that it is \emph{dominant} (see~\cite{gostz-maxvol-2010} for more details) in the rows and columns which it occupies, i.e.
\begin{equation}\label{eq:dom}
  \left| \sum_{t} \left[ A(\I_t,\J_s) \right]^{-1} A(\I_t,j) \right| \leq 1, 
  \qquad
  \left| \sum_{s} A(i,\J_s) \left[ A(\I_t,\J_s) \right]^{-1} \right| \leq 1.
\end{equation}

In~\cite{ot-ttcross-2010} it is shown that if a tensor $A$ is exactly given by~\eqref{eq:tt} with TT--ranks $r_k,$ it is recovered from $\O(d n r^2)$ tensor entries%
\footnote{We always assume $n_1=n_2=\ldots=n_d=n$ and $r_1=\ldots=r_{d-1}=r$ in complexity estimates} 
by the following formula.
\begin{equation}\label{eq:ii}
 \begin{split}
 A(i_1,\ldots,i_d) & = \sum_{\s,\t} A(i_1,\I_{t_1}^{> 1}) \left[ A(\I_{s_1}^{\leq 1},\I_{t_1}^{>1})\right]^{-1}  A(\I_{s_1}^{\leq 1},i_2,\I^{>2}_{t_2})  \ldots   A(\I^{\leq d-1}_{s_{d-1}},i_d)
 \\ & = \sum_{\s,\t} \prod_{k=1}^d A(\I^{\leq k-1}_{s_{k-1}},i_k,\I^{>k}_{t_k}) \left[A(\I^{\leq k}_{s_k},\I^{>k}_{t_k})\right]^{-1},
 \end{split}
\end{equation}
where $\I^{\leq k}_{s_k}$ and $\I^{>k}_{t_k}$ denote the positions of $r_k$ rows and columns in the $k$--th unfolding $A^{\{k\}}.$
To unify the notation, we introduce the empty border sets $\I^{\leq 0}=\emptyset$ and $\I^{>d}=\emptyset.$
We denote submatrices on the intersection of interpolation crosses as follows
\begin{equation}\nonumber
  \left[A(\I^{\leq k}_{s_k},\I^{>k}_{t_k})\right]_{t_k,s_k=1}^{r_k}
= \left[A(\I^{\leq k},\I^{>k})\right] = A_k,
\qquad A_k^{-1} = B^{[k]}.
\end{equation}
Throughout the paper we assume that TT--ranks of~\eqref{eq:tt} and~\eqref{eq:ii} are the same, i.e., 
sets $\I^{\leq k}=\{\I^{\leq k}_{1},\ldots,\I^{\leq k}_{r_k}\}$ and $\I^{>k}=\{\I^{>k}_{1},\ldots,\I^{>k}_{r_k}\}$  have $r_k$ elements each and $A_k=\left[A(\I^{\leq k},\I^{>k})\right]$ is $r_k \times r_k$ matrix, where $r_1,\ldots,r_{d-1}$ are TT--ranks of~\eqref{eq:tt}.
When a~\emph{choice} of $\I^{\leq k},\I^{>k}$ is considered, it means that we choose $r_k$ `left' and `right' multiindices $i_{\leq k}\in\I^{\leq k},$ $i_{>k}\in\I^{>k}.$


In~\eqref{eq:qm}, $\| \,\cdot\, \|_2$ denotes the~\emph{spectral} norm of a matrix, and $\|\,\cdot\,\|_C$ denotes the~\emph{Chebyshev} norm, also known as \emph{uniform}, \emph{supremum},  $\|\,\cdot\,\|_{\infty}$--norm, or the maximum entry in modulus.
For tensors Chebyshev and Frobenius norms are defined as follows
$$
|A|     = \|A\|_C = \max_{i_1,\ldots,i_d}|A(i_1,\ldots,i_d)|, \qquad
\|A\|^2 = \|A\|_F^2 = \sum_{i_1,\ldots,i_d}|A(i_1,\ldots,i_d)|^2.
$$

\section{Maximum--volume principle in higher dimensions} \label{sec:mvol}
We consider a tensor $A$ which is approximated by the TT format as follows
\begin{equation}\label{eq:aa}
 \begin{split}
 A(i_1,\ldots,i_d) & \approx X(i_1,\ldots, i_d) =  \sum_{\s} X^{(1)}(i_1,s_1) X^{(2)}(s_1,i_2,s_2) \ldots X^{(d)}(s_{d-1},i_d),
 \\ & \qquad |A-X| \leq E_C, \qquad \|A-X\| \leq E_F,
 \end{split}
\end{equation}
where $E_C$ and $E_F$ are known or estimated from computations or theoretical properties of $A.$
We apply~\eqref{eq:im} to $k$--th unfolding and write the cross interpolation
\begin{equation}\nonumber
 \begin{split}
  A^{\{k\}}(i_{\leq k},i_{>k}) & \approx
 \tilde A^{\{k\}}(i_{\leq k},i_{>k}) 
 = 
     \sum_{s_k,t_k} A^{\{k\}}(i_{\leq k},\I^{>k}_{t_k})  
     \left[A(\I^{\leq k}_{s_k},\I^{>k}_{t_k})\right]^{-1} 
     A^{\{k\}}(\I^{\leq k}_{s_k},i_{>k}).
 \end{split}
\end{equation}
For $\left[\I^{\leq k},\I^{>k}\right] = \maxvol A^{\{k\}}$ the accuracy is estimated by~\eqref{eq:qm} as follows
\begin{equation}\nonumber 
 \begin{split}
   \| A^{\{k\}} - \tilde A^{\{k\}} \|_C & \leq (r_k+1)^{\phantom{2}}    \| A^{\{k\}} - X^{\{k\}} \|_2 \leq (r_k+1)^{\phantom{2}}    \| A^{\{k\}} - X^{\{k\}} \|_F, \\
   \| A^{\{k\}} - \tilde A^{\{k\}} \|_C & \leq (r_k+1)^2                \| A^{\{k\}} - X^{\{k\}} \|_C.
  \end{split}
\end{equation}

We can safely omit the superscript for unfoldings when we use the pointwise notation, since the grouping of indices clearly defines the shape of the resulted matrix.
The equation for the unfolding is recast for the tensor as follows 
\begin{equation}\label{eq:qt}
 \begin{split}
  A(i_1,\ldots,i_d) & = \sum_{s_k,t_k} A(i_1,\ldots,i_k,\I^{>k}_{t_k}) 
                        B^{[k]}_{t_k,s_k} 
                        A(\I^{\leq k}_{s_k},i_{k+1},\ldots,i_d) 
   + E(i_1,\ldots,i_d),
\\ & \quad
   | E |  \leq (r_k+1)^{\phantom{2}}    E_F, \qquad
   | E |  \leq (r_k+1)^2                E_C, \qquad
   B^{[k]} = A_k^{-1}.
 \end{split}
\end{equation}
The interpolation step splits a $d$--tensor into a `product' of two tensors, which have $k$ and $d-k$ free indices, respectively.
The same splitting is done in~\cite{ot-tt-2009}, where a \emph{Tree--Tucker} format (later recast as the tensor train format) has been proposed to break the curse of dimensionality.
In~\cite{ot-tt-2009} the quasioptimality of the approximations computed by the proposed TT--SVD algorithm is shown.
Similarly, we estimate the accuracy of the interpolation--based formula~\eqref{eq:ii}.

\begin{lemma}\label{lem1}
If a tensor $A$ satisfies~\eqref{eq:aa}, then for any $k=1,\ldots,d-1$ it holds
\begin{equation}\label{eq:q1}
 \begin{split}
  A(\I^{\leq k-1}, i_k,i_{k+1},\I^{> k+1}) 
     & = \sum_{s_k,t_k} A(\I^{\leq k-1},i_k,\I^{>k}_{t_k}) 
         B^{[k]}_{t_k,s_k} 
         A(\I^{\leq k}_{s_k},i_{k+1},\I^{> k+1}) 
      \\ & + E(\I^{\leq k-1}, i_k,i_{k+1},\I^{> k+1}),
	\\ & \quad
   | E | \leq (r_k+1)^{\phantom{2}}    E_F,  \qquad
   | E | \leq (r_k+1)^2                E_C.
  \end{split}			
\end{equation}
\end{lemma}
\begin{proof} In~\eqref{eq:qt} we reduce free indices $i_{\leq k-1}$ to the subset $\I^{\leq k-1}$ and similarly $i_{>k+1}$ to $\I^{>k+1}.$ \end{proof}

\begin{lemma}\label{lem2}
 If a tensor $A$ satisfies~\eqref{eq:aa}, and for some $1\leq p<k<q\leq d$ for subtensors 
 $$
 A_\lft=\left[A(\I^{\leq p-1}, i_{p:k},\I^{> k})\right], \qquad
 A_\rgt=\left[A(\I^{\leq k}, i_{k+1:q},\I^{> q})\right], 
 $$
 it holds $A_\lft=T_\lft + E_\lft$ and $A_\rgt=T_\rgt + E_\rgt$ with $|E_\lft|\leq \eps |A|$ and $|E_\rgt| \leq \eps |A|,$ then
 \begin{equation}\label{eq:q2}
   \begin{split}
    A(\I^{\leq p-1}, i_{p:q},\I^{> q}) 
       & = \sum_{s_k,t_k} T_\lft(\I^{\leq p-1}, i_{p:k},\I^{> k}_{t_k}) 
           B^{[k]}_{t_k,s_k} 
           T_\rgt(\I^{\leq k}_{s_k}, i_{k+1:q},\I^{> q})
      \\ & + \E(\I^{\leq p-1}, i_{p:q},\I^{> q}),
      \\
    \frac{|\E|}{|A|}  \leq (2 + \eps\kappa_k) \eps r_k  & + \frac{|E|}{|A|},   \qquad \kappa_k = r_k |A| |A_k^{-1}|, \quad A_k= \left[ A(\I^{\leq k},\I^{>k}) \right], 
   \end{split}
 \end{equation}
where $|E|$ is estimated by~\eqref{eq:q1}.
\end{lemma}
\begin{proof}
Like in the previous lemma, by taking the subtensor in~\eqref{eq:qt} we obtain
\begin{equation}\nonumber
   \begin{split}
    A(\I^{\leq p-1}, i_{p:q},\I^{> q}) 
       & = \sum_{s_k,t_k}  A_\lft(\I^{\leq p-1}, i_{p:k},\I^{> k}_{t_k}) 
            B^{[k]}_{t_k,s_k} 
            A_\rgt(\I^{\leq k}_{s_k}, i_{k+1:q},\I^{> q})
      + E(\I^{\leq p-1}, i_{p:q},\I^{> q}),
   \end{split}
 \end{equation}
where $|E|$ is estimated by~\eqref{eq:q1}.
We have
\begin{equation}\nonumber
 \begin{split}
A_\lft B^{[k]} A_\rgt & = (T_\lft + E_\lft) B^{[k]} (T_\rgt + E_\rgt) 
 \\ & = T_\lft B^{[k]}  T_\rgt + A_\lft B^{[k]}  E_\rgt + E_\lft B^{[k]}  A_\rgt - E_\lft B^{[k]}  E_\rgt.
 \end{split}
\end{equation}
Since $A_k=\left[ A(\I^{\leq k},\I^{>k}) \right]$ is the maximum--volume submatrix in  $A^{\{k\}}=\left[A(i_{\leq k},i_{>k})\right],$ it dominates by~\eqref{eq:dom} in the corresponding rows and columns of the unfolding and \emph{a fortiori} in $A_\lft$ and $A_\rgt,$ i.e. $|A_\lft A_k^{-1}| \leq 1$ and $|A_k^{-1} A_\rgt| \leq 1.$
With $B^{[k]}=A_k^{-1}$ we have the following estimates
$$
 \begin{array}{c}
  |A_\lft B^{[k]} E_\rgt| \leq r_k |A_\lft B^{[k]} | \, |E_\rgt| \leq r_k \eps |A|,
  \qquad
  |E_\lft B^{[k]} A_\rgt|\leq r_k\eps |A|, \qquad \mbox{and}
  \\[1.1ex]
  | E_\lft B^{[k]} E_\rgt |  \leq r_k^2 \eps^2 | B^{[k]} | \, | A |^2 = r_k^2 \eps^2 | A_k^{-1} | \, | A |^2 = r_k \kappa_k \eps^2 |A|, 
 \end{array}
$$
which completes the proof.
\end{proof}

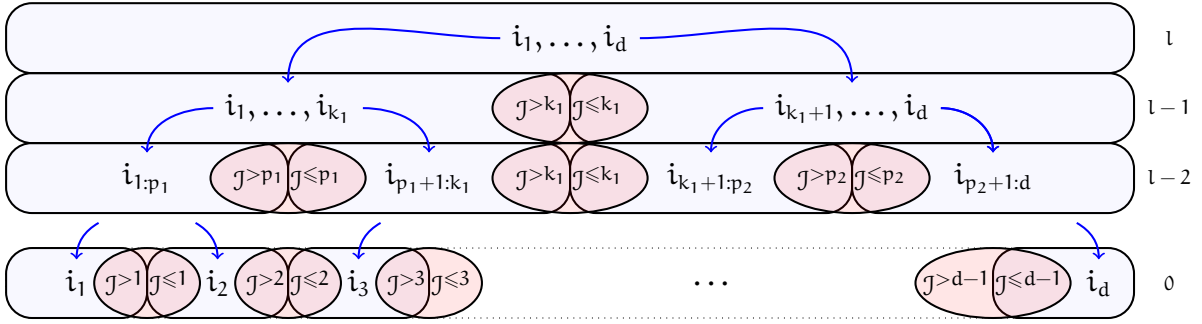
\begin{figure}[t]
 \begin{center}
  \resizebox{.99\textwidth}{!}{ \newdimen\TIKZx \TIKZx=.03\textwidth
 \newdimen\TIKZy \TIKZy=.03\textwidth
  \begin{tikzpicture}[x=\TIKZx,y=-\TIKZy]
   \tikzstyle{box}=[thick,thick,rounded corners=10,draw=black,fill=blue,fill opacity=.03];
   \tikzstyle{circ}=[thick,draw=black,fill=red,fill opacity=.10];
   
   \filldraw[box] (0,0) rectangle (32,2);
   \node (11) at (16,1) {$i_1,\ldots,i_d$};
   \node at (33,1) {\scriptsize $l$};
   
   \filldraw[box]   (0,2) rectangle (16,4);
   \filldraw[box]  (16,2) rectangle (32,4);
   \filldraw[circ] (16,3) circle[x radius=2.2\TIKZx, y radius=1.0\TIKZy];
   \node (21) at (8, 3) {$i_1,\ldots,i_{k_1}$}; 
   \node (22) at (24,3) {$i_{k_1+1},\ldots,i_d$};
   \node at (16,3) {\small $\I^{> k_1} \, \I^{\leq k_1} $};
   \node at (33,3) {\scriptsize $l-1$};
   
   \filldraw[box]   (0,4) rectangle ( 8,6);
   \filldraw[box]   (8,4) rectangle (16,6);
   \filldraw[box]  (16,4) rectangle (24,6);
   \filldraw[box]  (24,4) rectangle (32,6);
   \filldraw[circ] ( 8,5) circle[x radius=2.2\TIKZx, y radius=1.0\TIKZy];
   \filldraw[circ] (16,5) circle[x radius=2.2\TIKZx, y radius=1.0\TIKZy];
   \filldraw[circ] (24,5) circle[x radius=2.2\TIKZx, y radius=1.0\TIKZy];
   \node (31) at (4, 5) {$i_{1:p_1}$}; 
   \node (32) at (12,5) {$i_{p_1+1:k_1}$}; 
   \node (33) at (20,5) {$i_{k_1+1:p_2}$};
   \node (34) at (28,5) {$i_{p_2+1:d}$};
   \node at (8,5)  {\small $\I^{> p_1} \, \I^{\leq p_1} $};
   \node at (16,5) {\small $\I^{> k_1} \, \I^{\leq k_1} $};
   \node at (24,5) {\small $\I^{> p_2} \, \I^{\leq p_2} $};
   \node at (33,5) {\scriptsize $l-2$};
   
   \filldraw[box,thin,dotted,fill opacity=0.0]   (0,7) rectangle (32,9);
   \filldraw[box]   (0,7) rectangle ( 4,9);
   \filldraw[box]   (4,7) rectangle ( 8,9);
   \filldraw[box]   (8,7) rectangle (12,9);
   \filldraw[box]  (28,7) rectangle (32,9);
   \filldraw[circ] ( 4,8) circle[x radius=1.5\TIKZx, y radius=1.0\TIKZy];
   \filldraw[circ] ( 8,8) circle[x radius=1.5\TIKZx, y radius=1.0\TIKZy];
   \filldraw[circ] (12,8) circle[x radius=1.5\TIKZx, y radius=1.0\TIKZy];
   \filldraw[circ] (28,8) circle[x radius=2.2\TIKZx, y radius=1.0\TIKZy];
   \node at  (4,8)  {\small $\I^{>1} \, \I^{\leq1} $};
   \node at  (8,8)  {\small $\I^{>2} \, \I^{\leq2} $};
   \node at (12,8)  {\small $\I^{>3} \, \I^{\leq3} $};
   \node at (28,8)  {\small $\I^{>{d-1}} \, \I^{\leq d-1} $};
   \node (41) at (2,8)  {$i_1$};
   \node (42) at (6,8)  {$i_2$};
   \node (43) at (10,8) {$i_3$};
   \node (4d) at (31,8) {$i_d$};
   \node at (20,8) {$\ldots$};
   \node at (33,8) {\scriptsize $0$};
   
   \draw[->,blue,thick] (11.west)+(left:0.1mm)  to[out=left,in=up] (21.north);
   \draw[->,blue,thick] (11.east)+(right:0.1mm) to[out=right,in=up] (22.north);
   \draw[->,blue,thick] (21.west)+(left:0.1mm)  to[out=left,in=up] (31.north);
   \draw[->,blue,thick] (21.east)+(right:0.1mm) to[out=right,in=up] (32.north);
   \draw[->,blue,thick] (22.west)+(left:0.1mm)  to[out=left,in=up] (33.north);
   \draw[->,blue,thick] (22.east)+(right:0.1mm) to[out=right,in=up] (34.north);
   \draw[->,blue,thick] (22.east)+(right:0.1mm) to[out=right,in=up] (34.north);
   \draw[->,blue,thick] (41.north)+( 3mm,5mm) to[out=210,in=up]   (41.north);
   \draw[->,blue,thick] (42.north)+(-3mm,5mm) to[out=-30,in=up]   (42.north);
   \draw[->,blue,thick] (43.north)+( 3mm,5mm) to[out=210,in=up]   (43.north);
   \draw[->,blue,thick] (4d.north)+(-3mm,5mm) to[out=-30,in=up]   (4d.north);

  \end{tikzpicture}
 }
 \end{center}
\caption{Interpolation steps on the balanced dimension tree for Thm.~\ref{thm1}. 
        Rectangle boxes show the subtensors $\left[A(\I^{\leq p},i_{p:q},\I^{>q})\right]$ which are the building blocks of the decomposition. 
        Ellipses show the inverse matrices $B^{[k]}=A_k^{-1},$ $A_k=\left[A(\I^{\leq k},\I^{>k})\right],$ which do not carry free indices, but glue pairs of blocks together.}
\label{fig:t1}
\end{figure}

\begin{theorem}\label{thm1}
 If a tensor $A$ satisfies~\eqref{eq:aa}, and $E_F$ and/or $E_C$ are sufficiently small, then  $\tilde A$ given by~\eqref{eq:ii} with $\left[\I^{\leq k},\I^{>k}\right]=\maxvol A^{\{k\}}$ provides the accuracy
 \begin{equation}\label{eq1}
 \begin{split}
   |A-\tilde A| & \leq (2r+\kappa r+1)^{\lceil \log_2 d \rceil} (r+1)^{\phantom{2}} E_F, 
   \\
   |A-\tilde A| & \leq (2r+\kappa r+1)^{\lceil \log_2 d \rceil} (r+1)^2 E_C,
  \end{split}
 \end{equation}
 where $r=\max r_k,$ $\kappa=\max \kappa_k.$
 By `sufficiently small' we mean such values of $E_F$ and/or $E_C$ that the corresponding estimate provides $|A-\tilde A|/|A| < 1.$ 
 \end{theorem}
\begin{proof}
We will use the \emph{dimension tree} suggested in~\cite{ot-tt-2009}, see Fig.~\ref{fig:t1}. 
The interpolation step~\eqref{eq:qt} splits a given group of indices $i_p,\ldots,i_q$ in two parts $i_p,\ldots,i_k$ and $i_{k+1},\ldots,i_q,$ and introduces the auxiliary summation over the sets $\I^{\leq k}$ and $\I^{> k}$ at the point of splitting.  
No more than two auxiliary sets appear in each subtensor when the decomposition goes from the whole tensor down to leaves $\left[A(\I^{\leq k-1},i_k,\I^{>k})\right],$ which constitute~\eqref{eq:ii}. 
Leaves consist of the original entries of $A,$ therefore we have zero error at the ground level.
The interpolation error at the level $1$ is estimated by~\eqref{eq:q1} as follows
$$
\eps_1 = \frac{|E|}{|A|} = \min\left\{(r+1)\frac{E_F}{|A|}, \, (r+1)^2 \frac{E_C}{|A|} \right\}.
$$
When we move up by one level of the dimension tree, the error is amplified as shown by~\eqref{eq:q2}, and the relative error in Chebyshev norm propagates as follows
$$
\eps_{m+1}=(2+\eps_m \kappa) \eps_m r + \eps_1 \leq (2+\kappa)\eps_m r + \eps_m = (2r+\kappa r+1)\eps_m.  
$$
Here we use the inequality $\eps_m < 1$ provided by the assumption that $E_F$ and $E_C$ are sufficiently small.
Clearly, $\eps_l=(2r+\kappa r+1)^{l-1}\eps_1.$ 
For the balanced tree $2^l \leq 2d$ and $l \leq \lceil \log_2 d \rceil + 1,$ which completes the proof.
\end{proof}
\begin{remark}
 We are tempted to call $\kappa_k=r_k |A| |A_k^{-1}|$ the \emph{condition number} of the submatrix $A_k$ w.r.t. the Chebyshev norm. 
 Technically this is not correct, since in general $|A| \geq |A_k|$ and $\kappa_k \geq r_k |A_k| |A_k^{-1}| = \kappa_C(A_k).$
 However, in~\cite{gostz-maxvol-2010} it is shown that the ratio of the Chebyshev norms of a matrix and its maximum--volume submatrix is bounded as follows $|A|/|A_k| \leq 2r_k^2+r_k$ and often does not grow with rank.
 Therefore, $\kappa_k \leq (2r_k^2+r_k)\kappa_C(A_k),$ and usually $\kappa_k \simeq \kappa_C(A_k).$
 The similar formula with spectral norms appears in the pioneering paper on the cross interpolation~\cite[Eq. $(1.5)$]{gtz-psa-1997}.
\end{remark}

The splitting of indices in the balanced dimension tree was used to estimate the accuracy of the interpolation in the HT format~\cite{lars-htcross-2013}. 
The upper bound for the quasioptimality constant in the HT format is $\O(r^{2d}),$ where $r$ is the maximum representation rank.
Note that the upper bounds in~\eqref{eq1} do not \emph{necessarily} grow exponentially with~$d.$ 
More strict statement is possible if $\kappa$ remains bounded or grows moderately with $d$ as well, which certainly depends on the properties of a \emph{sequence} of $d$--tensors considered for $d=1,2,3,\ldots$
Such rigorous analysis is very important, but is beyond the scope of this paper.

The result of Thm.~\ref{thm1} can be interpreted as the \emph{existence} of a sufficiently good TT approximation computed from a few entries of a tensor by formula~\eqref{eq:ii}, provided that the accurate representation in the TT format~\eqref{eq:tt} is possible.
The coefficient $\O(r^{\lceil \log_2d \rceil + 2})$ can be also understood as upper bound for the ratio of the accuracy of the~\emph{best cross} interpolation~\eqref{eq:ii} and the \emph{best possible} accuracy of the approximation~\eqref{eq:tt} of the same TT--ranks.
Thm.~\ref{thm1} is \emph{constructive} and prescribes the choice of the interpolation sets $\I^{\leq k}, \I^{>k}$ to achieve the quasioptimal accuracy.
However, the actual~\emph{computation} of maximum--volume sets in unfoldings $A^{\{k\}}$ is impossible due to their restrictively large sizes.
In the next sections we consider the nested choice of the interpolation sets which reduces the search space.

\section{Nested maximum volume indices} \label{sec:emb}
In this section we switch to the ultimately unbalanced dimension tree, which splits indices one-by-one, see Fig.~\ref{fig:t2}. 
In~\cite{osel-tt-2011} this tree has been used to develop the TT--SVD algorithm which approximates a given $d$--tensor by the TT format.
We apply the same algorithm substituting the SVD approximation steps by the interpolation. 
As in the previous section, we estimate the accuracy of the resulted approximation w.r.t. the best possible approximation of the same TT--ranks.

Given a tensor $A=\left[A(i_1,\ldots,i_d)\right]$ that is approximated by the tensor train~\eqref{eq:aa}, we apply the interpolation formula~\eqref{eq:qt} and separate the rightmost index from the others as follows
\begin{equation}\nonumber
 A(i_1,\ldots,i_d)  = \sum_{\substack{s_{d-1}\\t_{d-1}}} A(i_{\leq d-1},\I^{> d-1}_{t_{d-1}}) 
          B^{[d-1]}_{t_{d-1},s_{d-1}}
          A(\I^{\leq d-1}_{s_{d-1}},i_d) 
   + E_{d-1}(i_1,\ldots,i_d),
\end{equation}
 where $\left[\I^{\leq d-1},\I^{>d-1}\right] = \maxvol \left[A(i_{\leq d-1},i_d)\right].$ 
Then we interpolate the subtensor with $d-1$ free indices and separate the rightmost free index as follows
\begin{equation}\nonumber
 \begin{split}
   A(i_{\leq d-1},\I^{> d-1}) & = \sum_{\substack{{s_{d-2}}\\{t_{d-2}}}}
     A(i_{\leq d-2},\I^{> d-2}_{t_{d-2}}) 
     B^{[d-2]}_{t_{d-2},s_{d-2}}
     A(\I^{\leq d-2}_{s_{d-2}},i_{d-1},\I^{>d-1}) 
    + E_{d-2}(i_{\leq d-1},\I^{> d-1}), 
 \end{split} 
\end{equation}
where $\left[\I^{\leq d-2},\I^{>d-2}\right] = \maxvol \left[A(i_{\leq d-2},i_{d-1}\I^{>d-1})\right].$
The elements of $\I^{>d-2}$ are now chosen not from all possible values of bi-index $\overline{i_{d-1}i_d}$ but from the reduced set $\overline{i_{d-1}\I^{>d-1}},$ where index $i_d$ is restricted to $r_{d-1}$ elements of $\I^{>d-1}.$
Hereinafter we omit the overline for the sake of clarity, since the use of comma in the pointwise notation is sufficient to show which indices are grouped together.
The maximum--volume subsets $\I^{>d-1}$ and $\I^{>d-2}$ are \emph{right--nested} (cf.~\cite{ot-ttcross-2010}) which means that $i_{>d-2} \in \I^{>d-2}$ leads to $i_{>d-1} \in \I^{>d-1}.$
As the interpolation develops further, it holds
\begin{equation}\label{eq:emb}
 i_{>k} \in \I^{>k} \:\Rightarrow\: 
 i_{>k+1} \in \I^{>k+1},
 \qquad k=d-1,\ldots,1.
\end{equation}

\begin{figure}[t]
 \begin{center}
  \resizebox{.99\textwidth}{!}{ \newdimen\TIKZx \TIKZx=.03\textwidth
 \newdimen\TIKZy \TIKZy=.03\textwidth
  \begin{tikzpicture}[x=\TIKZx,y=-\TIKZy]
   \tikzstyle{box}=[thick,thick,rounded corners=10,draw=black,fill=blue,fill opacity=.03];
   \tikzstyle{circ}=[thick,draw=black,fill=red,fill opacity=.10];
   
   \filldraw[box] (0,0) rectangle (32,2);
   \node (11) at (24,1) {$i_1,\ldots,i_d$};
   \node at (33,1) {\scriptsize $d-1$};
   
   \filldraw[box]   (0,2) rectangle (28,4);
   \filldraw[box]  (28,2) rectangle (32,4);
   \filldraw[circ] (28,3) circle[x radius=2.2\TIKZx, y radius=1.0\TIKZy];
   \node at (28,3) {\small $\I^{> d-1} \, \I^{\leq d-1} $};
   \node (21) at (20,3) {$i_1,\ldots,i_{d-1}$}; 
   \node (22) at (31,3) {$i_{d}$};
   \node at (33,3) {\scriptsize $d-2$};
   
   \filldraw[box]  (0, 4) rectangle (22,6);
   \filldraw[box]  (22,4) rectangle (28,6);
   \filldraw[box]  (28,4) rectangle (32,6);
   \filldraw[circ] (22,5) circle[x radius=2.2\TIKZx, y radius=1.0\TIKZy];
   \filldraw[circ] (28,5) circle[x radius=2.2\TIKZx, y radius=1.0\TIKZy];
   \node at (22,5) {\small $\I^{> d-2} \, \I^{\leq d-2} $};
   \node at (28,5) {\small $\I^{> d-1} \, \I^{\leq d-1} $};
   \node (31) at (16,5) {$i_1,\ldots,i_{d-2}$}; 
   \node (32) at (25,5) {$i_{d-1}$};
   \node (33) at (31,5) {$i_{d}$};
   \node at (33,5) {\scriptsize $d-3$};
   
   \filldraw[box,thin,dotted,fill opacity=0.0]   (0,7) rectangle (32,9);
   \filldraw[box]   (0,7) rectangle ( 4,9);
   \filldraw[box]   (4,7) rectangle ( 8,9);
   \filldraw[box]   (8,7) rectangle (12,9);
   \filldraw[box]  (28,7) rectangle (32,9);
   \filldraw[circ] ( 4,8) circle[x radius=1.5\TIKZx, y radius=1.0\TIKZy];
   \filldraw[circ] ( 8,8) circle[x radius=1.5\TIKZx, y radius=1.0\TIKZy];
   \filldraw[circ] (12,8) circle[x radius=1.5\TIKZx, y radius=1.0\TIKZy];
   \filldraw[circ] (28,8) circle[x radius=2.2\TIKZx, y radius=1.0\TIKZy];
   \node at  (4,8)  {\small $\I^{>1} \, \I^{\leq1} $};
   \node at  (8,8)  {\small $\I^{>2} \, \I^{\leq2} $};
   \node at (12,8)  {\small $\I^{>3} \, \I^{\leq3} $};
   \node at (28,8)  {\small $\I^{>{d-1}} \, \I^{\leq d-1} $};
   \node (41) at (2,8)  {$i_1$};
   \node (42) at (6,8)  {$i_2$};
   \node (43) at (10,8) {$i_3$};
   \node (4d) at (31,8) {$i_d$};
   \node at (20,8) {$\ldots$};
   \node at (33,8) {\scriptsize $0$};
   
   \draw[->,blue,thick] (11.west)+(left:0.1mm) to[out=left,in=up] (21.north);
   \draw[->,blue,thick] (11.east)+(right:0.1mm) to[out=right,in=up] (22.north);
   \draw[->,blue,thick] (21.west)+(left:0.1mm) to[out=left,in=up] (31.north);
   \draw[->,blue,thick] (21.east)+(right:0.1mm) to[out=right,in=up] (32.north);
   \draw[->,blue,thick] (22.south)+(down:0.1mm) to[out=down,in=up] (33.north);
   \draw[->,blue,thick] (41.north)+( 3mm,5mm) to[out=210,in=up]   (41.north);
   \draw[->,blue,thick] (42.north)+(-3mm,5mm) to[out=-30,in=up]   (42.north);
   \draw[->,blue,thick] (43.north)+( 0mm,5mm) to[out=down,in=up]  (43.north);
   \draw[->,blue,thick] (4d.north)+( 0mm,5mm) to[out=down,in=up]  (4d.north);

  \end{tikzpicture}
 }
 \end{center}
\caption{Interpolation steps on the unbalanced dimension tree for Thm.~\ref{thm2}, cf. Fig.~\ref{fig:t1}.}
\label{fig:t2}
\end{figure}

\begin{theorem}\label{thm2}
 If a tensor $A$ satisfies~\eqref{eq:aa}, then  $\tilde A$ given by~\eqref{eq:ii} with
 $$
 \left[\I^{\leq k},\I^{>k}\right]=\maxvol \left[A(i_{\leq k},i_{k+1}\I^{>k+1})\right], \qquad k=d-1,\ldots,1,
 $$
 provides the following accuracy
 \begin{equation}\label{eq2}
 \begin{split}
   |A-\tilde A| & \leq \frac{r^{d-1}-1}{r-1} (r+1)^{\phantom{2}} E_F, 
   \\
   |A-\tilde A| & \leq \frac{r^{d-1}-1}{r-1} (r+1)^{2} E_C.
  \end{split}
 \end{equation}
\end{theorem}
\begin{proof}
At the first level of the dimension tree the interpolation writes as follows
\begin{equation}\nonumber
A(i_1,i_2\I^{>2}) = \sum_{s_1,t_1} A(i_1,\I^{>1}_{t_1}) 
                 B^{[1]}_{t_1,s_1}
                 A(\I^{\leq1}_{s_1},i_2\I^{>2}) + E_1(i_1,i_2\I^{>2}),
\end{equation}
and since $\left[\I^{\leq1},\I^{>1}\right]=\maxvol\left[A(i_1,i_2\I^{>2})\right],$ it holds
\begin{equation}\nonumber
|E_1| \leq (r_1+1)^{\phantom{2}} E_F, \qquad |E_1| \leq (r_1+1)^2 E_C,
\end{equation}
which proves the statement of the theorem for $d=2.$
Suppose at the level $k$ of the tree it holds
\begin{equation}\nonumber
 \begin{split}
  A(i_{\leq k},\I^{>k}) & = \sum_{\substack{s_1\ldots s_{k-1}\\t_1\ldots t_{k-1}}}
   A(i_1,\I^{>1}_{t_1}) 
   B^{[1]}_{t_{1},s_{1}}
   \ldots 
   B^{[k-1]}_{t_{k-1},s_{k-1}}
   A(\I^{\leq k-1}_{s_{k-1}},i_k,\I^{>k})
 + \E_k(i_{\leq k},\I^{>k}), 
 \end{split}
\end{equation}
where  $|\E_k|\leq\frac{r^{k-1}-1}{r-1} |E_1|.$
Interpolation at the next level writes as follows
\begin{equation}\nonumber
 \begin{split}
  A(i_{\leq k},i_{k+1}\I^{>k+1}) 
     & = \sum_{s_k,t_k} A(i_{\leq k},\I^{>k}_{t_k}) 
     B^{[k]}_{t_{k},s_{k}}
     A(\I^{\leq k}_{s_k},i_{k+1}\I^{>k+1})
    + E_k(i_{\leq k+1},\I^{>k+1}).
 \end{split}  
\end{equation}
Using the previous equation we obtain
\begin{equation}\nonumber
 \begin{split}
  A(i_{\leq k+1},\I^{>k+1}) & = \sum_{\substack{s_1\ldots s_k\\t_1\ldots t_k}}
  A(i_1,\I^{>1}_{t_1}) 
   B^{[1]}_{t_{1},s_{1}}
  \ldots 
  B^{[k]}_{t_{k},s_{k}}
  A(\I^{\leq k}_{s_k},i_{k+1},\I^{>k+1}) 
 \\  
 & + \underbrace{\sum_{s_k,t_k} \E_k(i_{\leq k},\I^{>k}_{t_k}) 
  B^{[k]}_{t_{k},s_{k}}
 A(\I^{\leq k}_{s_k},i_{k+1},\I^{>k+1}) 
  + E_k(i_{\leq k+1},\I^{>k+1})}_{\E_{k+1}(i_{\leq k+1},\I^{>k+1})}.
 \end{split}
\end{equation}
Since $\I^{\leq k},\I^{>k}$ are chosen by the maximum--volume principle, we have
\begin{equation}\nonumber
 |E_k| \leq (r_1+1)^{\phantom{2}} E_F, \quad |E_k| \leq (r_1+1)^2 E_C, \qquad
 \left| \sum_{s_k} B^{[k]}_{t_k,s_k} A(\I^{\leq k}_{s_k},i_{k+1}\I^{>k+1}) \right| \leq 1,
\end{equation}
and it follows that
\begin{equation}\nonumber
 |\E_{k+1}| \leq r |\E_k| + |E_k| \leq r \frac{r^{k-1}-1}{r-1} |E_1| + |E_1| = \frac{r^k-1}{r-1} |E_1|.
\end{equation}
Substitution $k=d-1$ completes the proof.
\end{proof}

\begin{lemma}\label{lemi}
 If $\tilde A$ is given by~\eqref{eq:ii} and the interpolation sets are right--nested as shown by~\eqref{eq:emb}, then for all $k=1,\ldots,d-1$ it holds
 $$
 \tilde A(i_1,\ldots,i_k,\I^{>k}) = \sum_{\substack{s_1\ldots s_{k-1}\\t_1\ldots t_{k-1}}}
  A(i_1,\I^{>1}_{t_1})
  B^{[1]}_{t_1,s_1}
  \ldots 
  B^{[k-1]}_{t_{k-1},s_{k-1}}
  A(\I^{\leq k-1}_{s_{k-1}},i_k,\I^{>k}).
 $$
\end{lemma}
\begin{proof}
 To prove the statement of the lemma for $k=d-1,$ in~\eqref{eq:ii} we restrict $i_{>d-1}$ to $\I^{>d-1}.$ 
 The last core reduces to 
 $$
 \left[A(\I^{\leq d-1},i_d)\right]_{i_d\in\I^{>d-1}} = \left[A(\I^{\leq d-1},\I^{>d-1})\right] = A_{d-1},
 $$
 and cancels out with the neighboring matrix $B^{[d-1]}.$

 Suppose that the statement holds for $k=p+1,$ i.e.
 $$
 \tilde A(i_1,\ldots,i_{p+1},\I^{>p+1}) = \sum_{\substack{s_1\ldots s_{p}\\t_1\ldots t_{p}}}
 A(i_1,\I^{>1}_{t_1}) 
  B^{[1]}_{t_1,s_1}
 \ldots 
  B^{[p]}_{t_{p},s_{p}}
 A(\I^{\leq p}_{s_p},i_{p+1},\I^{>p+1}).
 $$
 Consider this equation for $i_{>p}\in\I^{>p},$ that by~\eqref{eq:emb} assumes $i_{>p+1}\in\I^{>p+1}.$ 
 The rightmost core reduces as follows 
 $$
 \left[A(\I^{\leq p},i_{p+1},\I^{>p+1})\right]_{i_{>p}\in\I^{>p}} = \left[A(\I^{\leq p},\I^{>p})\right] = A_p,
 $$
 and cancels out with $B^{[p]}.$ 
 This proves the statement for $k=p$ and the lemma by recursion. 
\end{proof}

Since $\left[\I^{\leq k},\I^{>k}\right]=\maxvol\left[A(i_{\leq k},i_{k+1}\I^{>k+1})\right],$ the quasioptimal estimate~\eqref{eq:qt} holds for the entries of this subtensor only.
However, $A_k=\left[A(\I^{\leq k},\I^{>k})\right]$ is nonsingular and we can interpolate the whole unfolding $A^{\{k\}}$ by the cross based on $A_k$ with some (presumably worse) accuracy estimate
\begin{equation}\label{eq:ihat}
  A(\i)  = \sum_{s_k,t_k} A(i_1,\ldots,i_k,\I^{>k}_{t_k}) 
            B^{[k]}_{t_k,s_k}
           A(\I^{\leq k}_{s_k},i_{k+1},\ldots,i_d) 
   + \hat E_k(\i).
\end{equation}
The following theorem estimates the accuracy of the same interpolation $\tilde A$ as in the previous theorem w.r.t. the errors $|\hat E_k|$ in~\eqref{eq:ihat}.

\begin{theorem}\label{thm3}
Under the conditions of Thm.~\ref{thm2} assume additionally that the interpolation~\eqref{eq:ihat} provides sufficiently good accuracy $\hat\eps = \max |\hat E_k|/|A|.$ 
Then 
 \begin{equation}\label{eq3}
   |A-\tilde A|  \leq \frac{dr\hat\eps}{1-d\kappa r \hat\eps} |A|,
 \end{equation}
where $\kappa$ is defined in~\eqref{eq1}.
By `sufficiently small' here we mean such $\hat\eps,$ that the denominator of~\eqref{eq3} does not approach zero.
\end{theorem}
\begin{proof}
The interpolation sets~\eqref{eq:emb} have been constructed from right to left according to the dimension tree on Fig.~\ref{fig:t2}.
In order to estimate the accuracy we separate indices one-by-one with the interpolation~\eqref{eq:ihat} proceeding from left to right.
We begin with
\begin{equation}\nonumber
  A(\i) = \sum_{s_1,t_1} 
         A(i_1,\I^{>1}_{t_1}) 
         B^{[1]}_{t_1,s_1}
         A(\I^{\leq 1}_{s_1},i_{>1}) + \hat E_1(\i),
\end{equation}
and $|\E_1|=|\hat E_1|\leq\hat\eps |A|.$ 
On the second step we write
\begin{equation}\nonumber
  A(\i) = \sum_{s_2,t_2} 
    A(i_1,i_2,\I^{>2}_{t_2}) 
    B^{[2]}_{t_2,s_2}
    A(\I^{\leq 2}_{s_2},i_{>2}) + \hat E_2(\i).
\end{equation}
We restrict $i_1$ to $\I^{\leq 1}$ and substitute the result into the previous equation. 
\begin{equation}\nonumber
 \begin{split}
  A(\i) & = \sum_{\substack{s_1,s_2\\t_1,t_2}}
      A(i_1,\I^{>1}_{t_1}) 
      B^{[1]}_{t_1,s_1}
      A(\I^{\leq1}_{s_1},i_2,\I^{>2}_{t_2}) 
      B^{[2]}_{t_2,s_2}
      A(\I^{\leq 2}_{s_2},i_{>2})  
  \\ & +  \underbrace{\sum_{s_1,t_1} 
           A(i_1,\I^{>1}_{t_1}) 
           B^{[1]}_{t_1,s_1}
          \hat E_2(\I^{\leq1}_{s_1},i_{>1}) + \E_1(\i)}_{\E_2(\i)}.
 \end{split}
\end{equation}
Since $\left[\I^{\leq1},\I^{>1}\right]=\maxvol\left[A(i_1,i_2\I^{>2})\right],$ submatrix $A_1$ dominates in the corresponding rows 
$
\left|\sum_{t_1} A(i_1,\I^{>1}_{t_1}) B^{[1]}_{t_1,s_1}\right|\leq1,
$
and therefore $|\E_2|\leq (r+1)\hat\eps |A|.$

The third interpolation step writes as follows
\begin{equation}\nonumber
  A(\i) = \sum_{s_3,t_3} 
         A(i_{\leq2},i_3,\I^{>3}_{t_3}) 
         B^{[3]}_{t_3,s_3}
         A(\I^{\leq 3}_{s_3},i_{>3}) + \hat E_3(\i).
\end{equation}
Again, 
we restrict $i_{\leq 2}$ to $\I^{\leq 2}$ and substitute the result into the previous equation. 
\begin{equation}\nonumber
 \begin{split}
  A(\i) & = \sum_{\substack{s_1,s_2,s_3\\t_1,t_2,t_3}}
            A(i_1,\I^{>1}_{t_1}) 
            B^{[1]}_{t_1,s_1}
            \ldots  
            B^{[3]}_{t_3,s_3}
            A(\I^{\leq 3}_{s_3},i_{>3})  
             + \E_3(\i),
  \\
 \E_3(\i) & = \sum_{\substack{s_1,s_2\\t_1,t_2}}
 \underbrace{A(i_1,\I^{>1}_{t_1}) 
            B^{[1]}_{t_1,s_1}
            A(\I^{\leq1}_{s_1},i_2,\I^{>2}_{t_2})}_{\tilde A(i_1,i_2,\I^{>2}_{t_2})} 
      B^{[2]}_{t_2,s_2}
     \hat E_3(\I^{\leq2}_{s_2},i_{>2}) + \E_2(\i).
 \end{split}
\end{equation}

We need to estimate the norm of the matrix in front of $\hat E_3,$ avoiding the exponential amplification of the coefficient. 
To do this, we replace the `piece' of the interpolation train with the subtensor of $A.$ 
Since $A=\tilde A+\E,$ the same holds for the subtensors $A(i_1,i_2,\I^{>2})=\tilde A(i_1,i_2,\I^{>2})+\E(i_1,i_2,\I^{>2}),$
and using Lemma~\ref{lemi} we write
$$
 \sum_{s_1,t_1} A(i_1,\I^{>1}_{t_1}) 
 B^{[1]}_{t_1,s_1}
 A(\I^{\leq1}_{s_1},i_2,\I^{>2}) 
 = A(i_1,i_2,\I^{>2}) - \E(i_1,i_2,\I^{>2}).
$$
Substituting this into the previous equation, we use the domination of the maximum--volume submatrix $A_2$ to write
$
\left| \sum_{t_2} A(i_{\leq 2},\I^{>2}_{t_2}) B^{[2]}_{t_2,s_2}\right| \leq 1 
$
and obtain
$$
|\E_3|\leq (2r+1) \hat\eps |A| + \kappa r \hat\eps |\E|.
$$
In further interpolation steps the error accumulates similarly.
Finally,
$$
|\E| = |\E_{d-1}| 
 \leq dr \hat\eps |A| + d \kappa r \hat\eps |\E|,
$$
which completes the proof.
\end{proof}

Theorems~\ref{thm2} and~\ref{thm3} estimate the accuracy of the interpolation formula~\eqref{eq:ii} with the same interpolation sets.
In Thm.~\ref{thm2} the quasioptimality result is proven with the coefficient $\O(r^d),$ which is much larger than the one in~\eqref{eq1}, cf. the coefficient~$\O(r^{2d})$ in~\cite{lars-htcross-2013}.
Since the coefficient in~\eqref{eq2} grows exponentially with the dimension, it can be hardly used in the real estimates. 
The result of Thm.~\ref{thm3} improves the estimate of Thm.~\ref{thm2} provided the errors $|\hat E_k|$ in~\eqref{eq:ihat} do not grow exponentially with $d.$
In general we cannot provide such upper bound for $|\hat E_k|.$ 
The estimate~\eqref{eq3} is useful in special cases when the theoretical or numerical estimates available for the errors $|\hat E_k|$ are bounded or grow moderately with $d.$

Note that the nestedness of the interpolation sets is essential in the proof of Thm.~\ref{thm3}. 
The result of Thm.~\ref{thm3} cannot be generalized to the `fully' maximum--volume case described in Thm.~\ref{thm1}.

\section{Two--side nestedness and the interpolation property}  \label{sec:emb2}
In this section we consider the interpolation~\eqref{eq:ii} where both left and right interpolation sets are nested, i.e. for all valid $k$ it holds
\begin{equation}\label{eq:emb2}
  i_{>k} \in \I^{>k} \:\Rightarrow\:  i_{>k+1} \in \I^{>k+1}, \qquad
  i_{\leq k} \in \I^{\leq k} \:\Rightarrow\: i_{\leq k-1} \in \I^{\leq k-1}.
\end{equation}

The naive way to construct such sets is to run the right--to--left interpolation pass explained in Sec.~\ref{sec:emb} and keep the right sets $\I^{>k}$ only. 
The left sets $\I^{\leq k}$ are computed by the left--to--right interpolation pass which separates the index $i_1,$ then $i_2,$ etc.
We obtain
\begin{equation}\nonumber
   \left[\J^{\leq k},\I^{>k}\right]  = \maxvol\left[A(i_{\leq k},i_{k+1}\I^{>k+1})\right], \quad  
   \left[\I^{\leq k},\J^{>k}\right]  = \maxvol\left[A(\I^{\leq k-1}i_k,i_{>k})\right].
\end{equation}
Note that $\left[A(\I^{\leq k},\I^{>k})\right]$ is not necessarily the maximum--volume submatrix neither in the subtensor $\left[A(i_{\leq k},i_{k+1}\I^{>k+1})\right],$  nor in $\left[A(\I^{\leq k-1}i_k,i_{>k})\right],$ nor even in their intersection $\left[A(\I^{\leq k-1}i_k,i_{k+1}\I^{>k+1})\right].$
Therefore, we cannot use~\eqref{eq:qm} to estimate the accuracy of~\eqref{eq:ii} with these interpolation sets.
Due to the restrictive sizes, the computation of the maximum volume submatrix is impossible even with implied nestedness.
To make the problem tractable, we should further reduce the search space --- the practical recipes will be discussed in the next section.

If both left and right interpolation sets are nested, Eq.~\eqref{eq:ii} is indeed the \emph{cross interpolation formula}, as shown by the following theorem.
\begin{theorem}\label{thmi}
 For a tensor $A,$ the approximation $\tilde A$ given by~\eqref{eq:ii} with indices $\I^{\leq k},\I^{>k}$ satisfying~\eqref{eq:emb2}, is exact on the positions of all entries evaluated in a tensor
 \begin{equation}\label{eqi}
  A(\I^{\leq k-1},i_k,\I^{>k}) = \tilde A(\I^{\leq k-1},i_k,\I^{>k}), \qquad k=1,\ldots,d.
 \end{equation}
\end{theorem}
\begin{proof} 
 It is sufficient to repeat the arguments from the proof of Lemma~\ref{lemi} for the left and right interpolation sets.
\end{proof}

A $m\times n$ matrix $A$ of rank $r$ is defined by $(m+n)r-r^2$ parameters, e.g. by $mr+nr+r$ elements of the SVD decomposition $A=USV^\trans$ minus $r(r+1)$ normalization constraints $U^\trans U=I,$ $V^\trans V=I.$
The cross interpolation formula~\eqref{eq:im} recovers a rank--$r$ matrix from $(m+n)r-r^2$ entries, if a submatrix $\left[A(\I,\J)\right]$ is nonsingular.
If $\rank A>r,$ formula~\eqref{eq:im} provides the approximation $\tilde A,$ which is exact on $(m+n)r-r^2$ positions of a matrix.
This fact is generalized to the tensor case by the following theorem.
\begin{theorem}
 A tensor $A$ with mode sizes $n_1,\ldots,n_d$ and TT--ranks $r_1,\ldots,r_{d-1}$ is defined by 
 $$
 s=\sum_{k=1}^d r_{k-1} n_k r_k - \sum_{k=1}^{d-1} r_k^2
 $$
 parameters.
 If the left and right interpolation sets satisfy~\eqref{eq:emb2}, and the matrices $A_k,$ $k=1,\ldots,d-1,$ are nonsingular, formula~\eqref{eq:ii} recovers $A$ from exactly $s$ entries.
 If a tensor $A$ is not given by~\eqref{eq:tt} exactly, formula~\eqref{eq:ii} interpolates it on at least $s$ positions.
\end{theorem}
\begin{proof}
 The first statement is proven in~\cite[Prop. A.3]{ushmaev-tt-2013}. 
 Taking into account the result of Thm.~\ref{thmi}, the second and the third statements require to calculate the total number of tensor entries in all subtensors in~\eqref{eqi}.
 Each block $\left[A(\I^{\leq k-1},i_k,\I^{>k})\right]$ consists of $r_{k-1}n_kr_k$ elements of a tensor, but some entries contribute to more than one block.
 For example, if~\eqref{eq:emb2} holds, subtensors $\left[A(i_1,\I^{>1})\right]$ and $\left[A(\I^{\leq 1},i_2,\I^{>2})\right]$ intersect by the submatrix $A_1=\left[A(\I^{\leq 1},\I^{>1})\right],$ which has $r_1^2$ elements.
 Similarly, $\left[A(\I^{\leq k-1},i_k,\I^{>k})\right]$ and $\left[A(\I^{\leq k},i_{k+1},\I^{>k+1})\right]$ have $r_k^2$ common elements in the submatrix $A_k.$
 
 The common elements of  $\left[A(\I^{\leq k},i_{k+1},\I^{>k+1})\right]$ and $\left[A(\I^{\leq p-1},i_p,\I^{>p})\right]$ are described by the following conditions 
 $$
 i_{\leq p-1}\in\I^{\leq p-1}, \quad i_{\leq k} \in \I^{\leq k}, \qquad i_{>p}\in\I^{>p}, \quad i_{>k+1}\in\I^{>k+1}.
 $$
 If $p<k,$ they are reduced by~\eqref{eq:emb2} to $i_{\leq k} \in \I^{\leq k}$ and  $i_{>p}\in\I^{>p},$ and it holds
 $$
 \{i_{\leq k} \in \I^{\leq k}, i_{>p}\in\I^{>p}\} \subset 
 \{i_{\leq k} \in \I^{\leq k}, i_{>k}\in\I^{>k}\}.
 $$
 Therefore, all common entries of $\left[A(\I^{\leq k},i_{k+1},\I^{>k+1})\right]$ and $\left[A(\I^{\leq p-1},i_p,\I^{>p})\right]$ belong to $A_k$ for $p=1,\ldots,k-1.$
 The total number of entries which belong to more than one block equals $\sum_{k=1}^{d-1} r_k^2,$ which completes the proof.
\end{proof}

\section{Interpolation algorithms for matrices and tensors} \label{sec:alg}
We start this section with a short overview of the cross interpolation algorithms for matrices.
The idea of reconstruction and approximation of a matrix from several columns and rows by the~\emph{skeleton decomposition}~\eqref{eq:im} or the \emph{pseudoskeleton decomposition} $\tilde A=CGR,$ $C=\left[A(i,\J)\right],$ $R=\left[A(\I,j)\right],$  has been suggested by Goreinov and Tyrtyshnikov~\cite{gt-psa-1995}.
In~\cite{gtz-psa-1997} the accuracy of the pseudoskeleton approximation has been studied and it has been pointed out that a good cross should intersect by a well bounded submatrix.
The connection with the maximum--volume submatrix has been mentioned in~\cite{gtz-psa-1997}, and the maximum--volume principle has been presented in more detail in~\cite{gt-maxvol-2001}.

\begin{algorithm}[t] 
 \caption{Greedy cross interpolation algorithm for tensor trains} \label{algg}
 \begin{algorithmic}[1]
  \REQUIRE Function to compute entries of a tensor $A=\left[A(i_1,\ldots,i_d)\right]$
  \ENSURE Cross interpolation~\eqref{eq:ii} with the nested interpolation sets~\eqref{eq:emb2} 
  \STATE $\I^{\leq k}=\emptyset,$ $\I^{>k}=\emptyset,$ $k=1,\ldots,d,$ $\tilde A=0,$ $E=A$
  \WHILE{$|A-\tilde A|$ is not sufficiently small}
    \STATE Find a pivot $i^\new=(i^{\new}_1,\ldots,i^{\new}_d)$ s.t. $|E(i^\new_1,\ldots,i^\new_d)| \simeq |A-\tilde A|$
    \STATE Add $i^\new_{\leq k}$ to $\I^{\leq k},$ and $i^\new_{>k}$ to $\I^{>k},$ $k=1,\ldots,d-1$
    \STATE Update the interpolation $\tilde A$ by~\eqref{eq:ii}
  \ENDWHILE
 \end{algorithmic}
\end{algorithm}

The search for the maximum--volume submatrix \emph{per se} is an NP--hard problem~\cite{bartholdi-1982}. 
For practical computations, it is necessary to find a \emph{sufficiently good} submatrix reasonably fast.
The alternating direction algorithm has been proposed in~\cite{tee-cross-2000}, which adaptively increases the size of the interpolation cross following the maximum--volume principle at each step, and computes the approximation of a matrix in linear time w.r.t. the size.
The greedy algorithm of such kind, equivalent to the Gaussian elimination with partial pivoting, was then suggested by Bebendorf~\cite{bebe-2000}.
Due to its particular simplicity, it has become widely known as the \emph{adaptive cross approximation} (ACA).
In practical computations, ACA and similar methods with minimal information are liable to breakdowns, i.e. they may quit when a good approximation is not yet obtained.
A cheap remedy proposed in~\cite{sav-2006} is to check the accuracy on the random set of entries and restart the algorithm if necessary.\footnote{Published in English later as~\cite[Alg. 3]{ost-chem-2010}}
Another well--known sampling method is the $CUR$ algorithm of Mahoney et al~\cite{mahoney-cur-2006}, which is the pseudoskeleton $CGR$ decomposition where positions of the rows and columns are chosen randomly.

The accuracy of the maximum--volume cross approximation is estimated for any matrix~\cite{gt-maxvol-2001,schneider-cross2d-2010,gt-skel-2011}.
Algorithms which use a few elements (e.g. ACA) are \emph{heuristic} and construct the approximation which can be arbitrarily bad for other matrix elements.
The accuracy of such algorithms can be estimated in special cases, e.g. for matrices generated by asymptotically smooth functions on quasi--uniform grids, see~\cite{bebe-2000,tee-kron-2004}, cf.~\cite{tee-tensor-2003} in many dimensions.

The existing cross interpolation algorithms for tensors can be classified similarly.
The skeleton decomposition is generalized to the tensor case in~\cite{ot-ttcross-2010} by formula~\eqref{eq:ii}, where the submatrices $A_k=\left[A(\I^{\leq k},\I^{>k})\right]$ play the same role as $\left[A(\I,\J)\right]$ in~\eqref{eq:im}.
The `existence result' is generalized from the matrix case~\cite{gt-maxvol-2001,schneider-cross2d-2010,gt-skel-2011} to the TT case by Thm.~\ref{thm1}.
Algorithm proposed in~\cite{ot-ttcross-2010} approximates the maximum--volume positions in the ALS way, similarly to the one from~\cite{tee-cross-2000}.

A greedy cross interpolation algorithm for the TT format can be suggested similarly to the matrix case, see e.g. Alg.~\ref{algg}. 
Similarly to the ACA, Alg.~\ref{algg} relies on the interpolation property for the tensor trains, established by Thm.~\ref{thmi}.
On each step Alg.~\ref{algg} searches for a pivot $i^\new$  where the error of the current approximation is (quasi)maximum in modulus.
Then it adds the indices of $i^\new=\overline{i_{\leq k}^\new i_{>k}^\new}$ to all subsets $\I^{\leq k}$ and $\I^{>k},$ $k=1,\ldots,d-1,$ to maintain the two--side nestedness~\eqref{eq:emb2}.
The updated interpolation is exact on all \emph{lines} $(i^{\new}_1,\ldots,i^\new_{k-1},i_k,i^\new_{k+1},\ldots,i_d^\new),$ $i_k=1,\ldots,n_k$ $k=1,\ldots,d.$

The full pivoting in higher dimensions is impossible due to the curse of dimensionality, and we need cheaper alternatives  to find a new pivot and estimate the accuracy for the stopping criterion.
Following the tensor--CUR algorithm of Mahoney et al~\cite{mahoney-simax-2008} we can choose indices randomly.
Another approach is to choose the maximum in modulus element of the current residual among a randomly sampled set.
The third option is to choose the pivot from a \emph{restricted set} similarly to the ACA approach, check the accuracy of the approximation over a random set of entries, and restart if necessary, see~\cite[Alg. 3]{ost-chem-2010}.

The restricted pivoting set can naturally arise from the \emph{locality} requirement.
By this we mean that with a new pivot we should modify only a few interpolation sets $\I^{\leq k}$ and $\I^{>k}$ and increase only a few TT--ranks of the approximation, not all of them.
To put it differently, a pivoting algorithm should update only a few TT--cores of~\eqref{eq:ii} at each step, similarly to the ALS and DMRG algorithms introduced in quantum physics.

Following the DMRG algorithm, we choose a new pivot $i^\new$ in the DMRG \emph{supercore}  $A'=\left[A(\I^{\leq k-1},i_k, i_{k+1},\I^{>k+1})\right].$ 
This choice provides $i^\new_{\leq k-1}\in\I^{\leq k-1}$ and by~\eqref{eq:emb2} $i^\new_{\leq p}\in\I^{\leq p}$ for $p\leq k-1.$
Similarly, $i^\new_{>k+1}\in\I^{>k+1}$ and by nestedness $i^\new_{>p}\in\I^{>p}$ for $p\geq k+1.$
When we add $i^\new_{\leq k}$ to $\I^{\leq k}$ and $i^\new_{>k}$ to $\I^{>k},$ the two--side nestedness~\eqref{eq:emb2} is preserved \emph{ipso facto.} 

\begin{algorithm}[t] 
 \caption{Greedy restricted cross interpolation algorithm for tensor trains} \label{algi}
 \begin{algorithmic}[1]
  \REQUIRE Function to compute entries of a tensor $A=\left[A(i_1,\ldots,i_d)\right]$
  \ENSURE Cross interpolation~\eqref{eq:ii} with the nested interpolation sets~\eqref{eq:emb2} 
  \STATE Choose $\I^{\leq k},$ $\I^{>k},$ $k=1,\ldots,d,$ which satisfy~\eqref{eq:emb2}, and compute $\tilde A$ by~\eqref{eq:ii}
  \WHILE{stopping criterion is not satisfied}
  \FOR[Left--to--right half--sweep]{$k=1,\ldots,d-1$}
    \STATE Apply the cross interpolation (e.g.~\cite[Alg. 3]{ost-chem-2010}) to the DMRG supercore matrix $\left[A(\I^{\leq k-1}i_k, i_{k+1}\I^{>k+1})\right],$ 
    using sets $\I^{\leq k},\I^{>k}$ as the initial guess, and compute~\eqref{eq:dmrgs}
    with  $\I^{\leq k} \subset \J^{\leq k}$ and  $\I^{> k} \subset \J^{> k}$ 
    \STATE Substitute $\I^{\leq k}$ and $\I^{>k}$ by the expanded sets $\J^{\leq k}$ and $\J^{>k}$  
  \ENDFOR
  \STATE Perform right--to-left half--sweep in the same way
  \ENDWHILE
 \end{algorithmic}
\end{algorithm}

The greedy algorithm with pivoting in $A'$ can be implemented as a simple modification of the cross interpolation algorithm TT--RC from~\cite{so-dmrgi-2011proc}.
The TT--RC algorithm is of the DMRG type, which means that it updates two neighboring TT--cores at a step, computing the matrix $A'$ in full.
The proposed Alg.~\ref{algi} substitutes this step with the cross interpolation and approximates
\begin{equation}\label{eq:dmrgs}
    A(\I^{\leq k-1}i_k, i_{k+1}\I^{>k+1}) 
           \approx \sum_{s_k,t_k} A(\I^{\leq k-1}i_k, \J_{t_k}^{>k}) 
           \left[A(\J_{s_k}^{\leq k},\J_{t_k}^{>k})\right]^{-1} 
           A(\J_{s_k}^{\leq k}, i_{k+1}\I^{>k+1}), 
\end{equation}
where $\J^{\leq k}$ and $\J^{>k}$ are computed by the matrix cross interpolation algorithm, s.t. 
\begin{equation}\nonumber 
   \left[\J^{\leq k},\J^{>k}\right]  \simeq \arg\max_{\substack{\I^{\leq k-1} \\ \I^{>k+1}}} \vol\left[A(\I^{\leq k-1}i_k,i_{k+1}\I^{>k+1})\right].
\end{equation}
The resulting algorithm requires $\O(dnr^2)$ evaluation of tensor elements and $\O(dnr^3)$ additional operations, i.e. scales linearly in the mode size and very moderately in the TT--rank.
The algorithm is rank--revealing, i.e. will not increase the TT--ranks of the approximation~\eqref{eq:ii} over the TT--ranks of a given tensor.

The greedy algorithms are not always good in practice, since the positions chosen as the initial guess may approximate the maximum--volume submatrices inaccurately and should be removed when the interpolation sets are sufficiently large.
Only a slight modification of Alg.~\ref{algi} is required to develop a non--greedy version.

\newpage
\section{Numerical experiments} \label{sec:num}
The numerical results have been obtained using the Iridis3 High Performance Computing Facility at the University of Southampton.\footnote{Iridis3 is based on Intel $2.4$ GHz processors, for more specifications see \href{http://cmg.soton.ac.uk/iridis}{cmg.soton.ac.uk/iridis}.}
Cross interpolation and auxiliary tensor train subroutines are written in \textsc{Fortran$90$} by the author.
The code was compiled using the Intel Composer and linked with \textsc{Lapack}/\textsc{Blas} subroutines provided with the {MKL} library.

In the experiments we use a very simple version of Alg.~\ref{algi}.
On each step (Line~4) we improve the current approximation by adding only one cross to $[\I^{\leq k},\I^{>k}]$. 
The position of the new cross is computed as follows.
First, a random sampling is performed on $r_{k-1}n_k+n_{k+1}r_{k+1}$ entries of the matrix $\left[A(\I^{\leq k-1}i_k, i_{k+1}\I^{>k+1})\right],$ and an element is chosen where the error of the current interpolation is maximum in modulus.
Then the residual for the row or column (for left and right half--sweep, resp.) which contains this element is evaluated, and the pivot $i^\new$ is chosen among its entries.
If pivot is not zero up to the machine precision, the obtained cross is added to interpolation sets $\I^{\leq k},\I^{>k}.$
If pivot is machine null, the rank $r_k$ is not increased.

The interpolation sets are always initialized by the index $(1,1,\ldots,1).$

\subsection{The quasioptimality coefficient}
\begin{figure}[t]
\hbox to \textwidth{\hfil
\includegraphics[width=.45\textwidth]{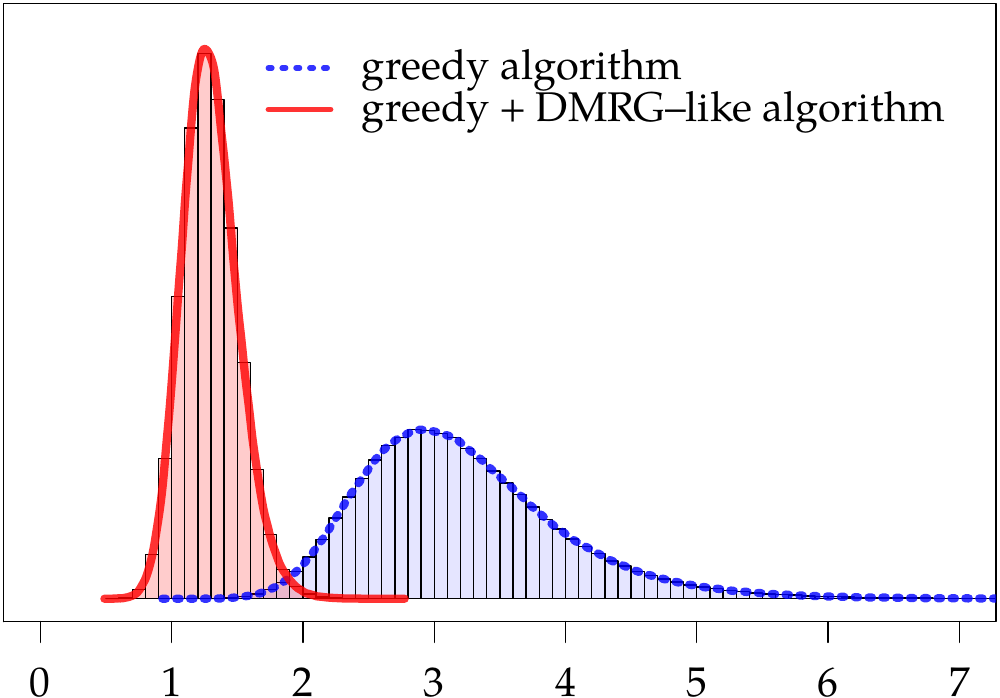} \hfil
\includegraphics[width=.45\textwidth]{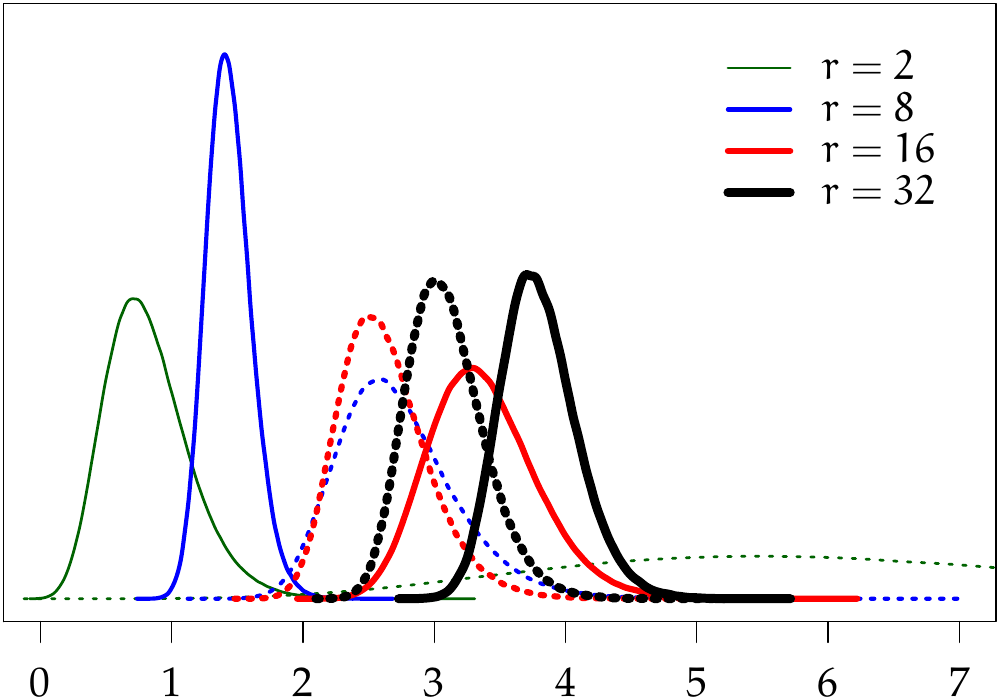} \hfil
}
\vskip 4mm
\hbox to \textwidth{\hfil
\includegraphics[width=.45\textwidth]{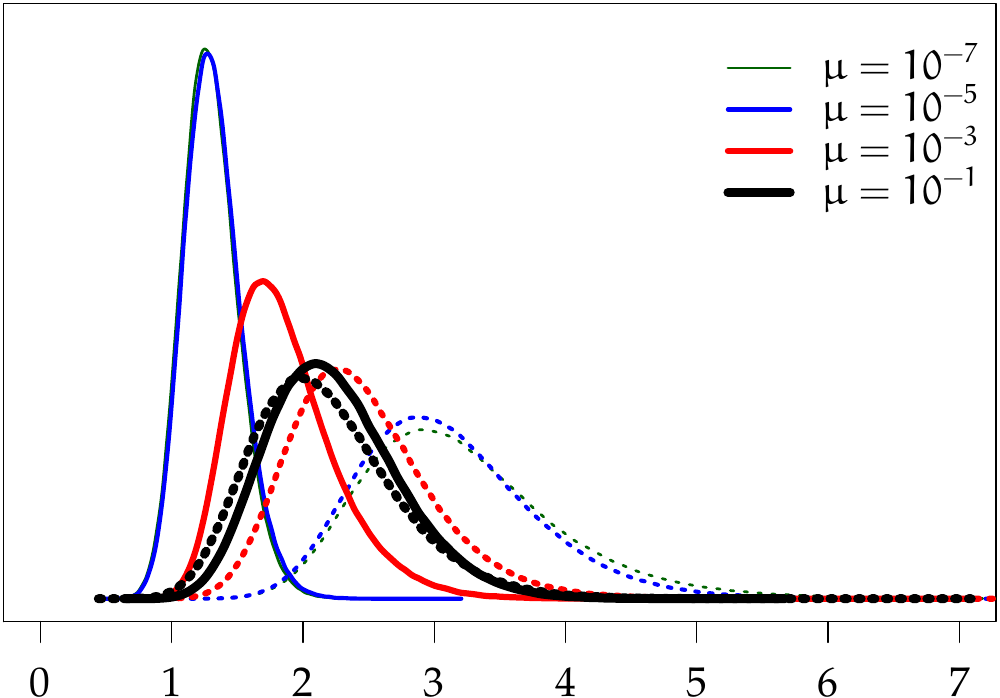} \hfil
\includegraphics[width=.45\textwidth]{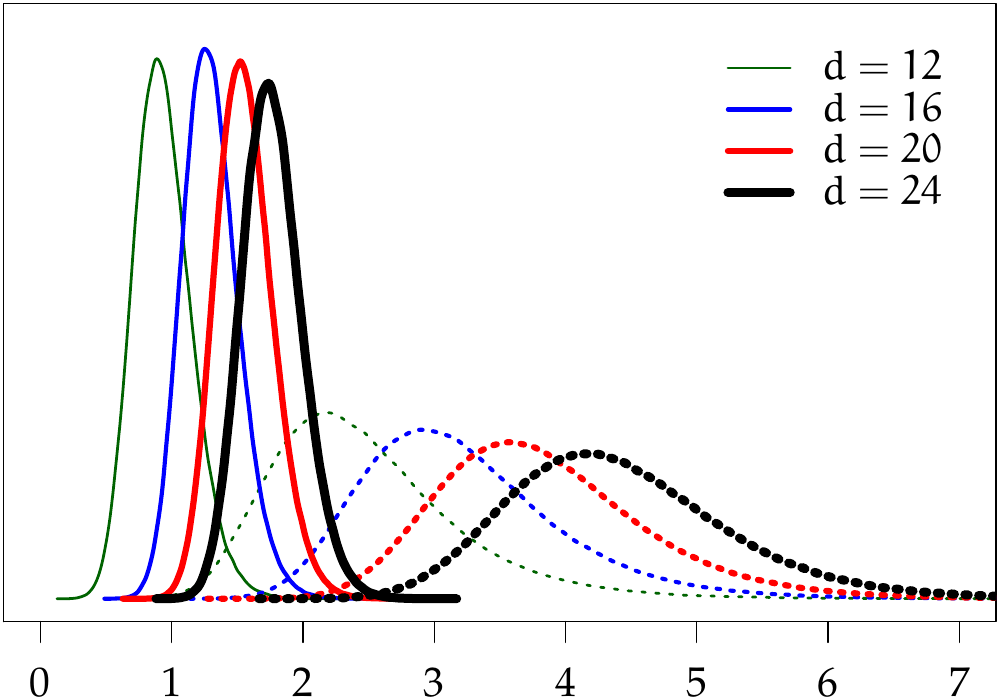} \hfil
}
\caption{Distribution of $\log_2 (|A-\tilde A|/|A-X|)$ for randomly generated tensors $A$ given by~\eqref{eq:rnd} and the cross interpolation $\tilde A,$ computed by Alg.~\ref{algi} (dashed lines), and additionally improved by the TT-RC algorithm from~\cite{so-dmrgi-2011proc} (solid lines). 
Dimension $d=16,$ mode size $n=2,$ noise level $\mu=10^{-7},$ rank $r=5,$ unless other value is shown on the graph.}
\label{fig:q}
\end{figure}

For a number of randomized experiments we measure the ratio between the accuracy of the approximation in the TT format~\eqref{eq:tt} and the cross interpolation~\eqref{eq:ii} with the same TT--ranks.
Given dimension $d,$ mode size $n=2,$ mode ranks $r$ and  \emph{noise level} $\mu,$ we consider the tensor 
\begin{equation}\label{eq:rnd}
 A = X + \mu R, \qquad |X|=1, \quad |R|=1,
\end{equation}
where $R$ is random and $X$ is given by the TT format~\eqref{eq:tt} with TT--ranks $r$ and random TT--cores.
All random elements are independently and uniformly distributed on the unit set and we seed them using the internal pseudorandom generator provided with the compiler.

We apply  Alg.~\ref{algi} to compute the initial cross interpolation $\tilde A_{\mathrm{greedy}}$ with TT--ranks not larger than $r.$ 
Then we run $10$ additional sweeps of the DMRG--like TT--RC algorithm~\cite{so-dmrgi-2011proc} to improve the positions of the interpolation crosses and obtain $\tilde A_{\mathrm{DMRG}}.$
Density distributions of the logarithm of the quasioptimality coefficient for $\tilde A_{\mathrm{greedy}}$ and $\tilde A_{\mathrm{DMRG}}$ are shown on Fig.~\ref{fig:q}.
The number of tests for each density distribution curve is at least $2^{20}.$

We note that for the randomly generated tensors, the quasioptimality coefficient is not very large.
For example, the top left graph on Fig.~\ref{fig:q} corresponds to $d=16$ and $r=5.$ 
The estimate of Thm.~\ref{thm1} provides the upper bound for the quasioptimality coefficient  $(2r+\kappa r+1)^{\lceil \log_2d \rceil} (r+1)^2 \geq 16^4 6^2 \geq 2^{21}.$
The computed value is
$$
\log_2(|A-\tilde A_{\mathrm{greedy}}|/|A-X|)=3.2\pm2.1, \qquad
\log_2(|A-\tilde A_{\mathrm{DMRG}}|/|A-X|)=1.3\pm0.6.
$$
Therefore, for the considered experiment the upper bound $2^{21}$ provided by Thm.~\ref{thm1} overestimates the actual value by a factor $\geq 2^{19.5}.$

It is important how the accuracy of the interpolation depends on the dimension $d$ and the TT--rank $r.$
The result of these experiments are shown in the right column of Fig.~\ref{fig:q}.
We see that the coefficient grows with rank and dimension slower than the upper bound~\eqref{eq1}.
For example, for $r=32$ the upper bound is $\gtrsim 2^{58},$ assuming $\kappa=1.$ The actual coefficient computed in the numerical experiments is of the order $2^3$ for $\tilde A_{\mathrm{greedy}}$ and $2^4$ for $\tilde A_{\mathrm{DMRG}}.$
Note that in this case the interpolation improved by the DMRG--like algorithm has worse accuracy than the interpolation returned by~Alg.~\ref{algi}.
This may be explained by the fact that the TT--RC algorithm has the truncation step which reduces TT--ranks and introduces a perturbation to the tensor.
The double--side nestedness~\eqref{eq:emb2} is not preserved during this step which may result in the loss of the interpolation property and deteriorate the accuracy. 
This emphasizes the importance of the interpolation property given by Thm.~\ref{thmi}. 

Finally, we analyze how the accuracy of the cross interpolation depends on the noise level $\mu.$
On the bottom left graph on Fig.~\ref{fig:q} we see that this parameter does not change the distribution significantly.
When $\mu\leq10^{-5},$ further reduction of the noise level has no effect on the distribution of the quasioptimality coefficient.

We summarize that for random tensors the accuracy of the computed cross interpolation behaves much better than the upper bound in~\eqref{eq1}.

\subsection{Speed and accuracy of the greedy interpolation algorithm}
\begin{table}[p]
 \def\eee{_{10}-}
 \def\p#1#2#3{\begin{tabular}{c} $#1$ \\ $#2$ \\ $#3$ \end{tabular}}
 \begin{center}
   \begin{tabular}{c|ccc|ccc}
         & \multicolumn{3}{c|}{$d=2^4$} & \multicolumn{3}{c}{$d=2^5$}  \\
    $r$  & $n=2^5$ & $n=2^7$ & $n=2^9$                                                                 & $n=2^5$ & $n=2^7$ & $n=2^9$  \\ \hline
    $ 6$ & \p{8\eee3 }{7\eee3 }{0.007} & \p{1\eee1 }{1\eee1 }{0.035}  &                                & \p{1\eee1 }{6\eee2 }{0.021} &                             &                           \\ \hline  
    $12$ & \p{2\eee5 }{3\eee6 }{0.033} & \p{7\eee4 }{2\eee4 }{0.14}  & \p{2\eee2 }{7\eee3 }{0.59}      & \p{6\eee5 }{9\eee6 }{0.097} & \p{3\eee3 }{1\eee3 }{0.37}  & \p{8\eee2 }{4\eee2 }{1.56}  \\ \hline  
    $18$ & \p{8\eee9 }{2\eee9 }{0.081} & \p{2\eee6 }{3\eee6 }{0.34}  & \p{7\eee5 }{1\eee5 }{1.45}      & \p{2\eee8 }{5\eee9 }{0.23}  & \p{2\eee5 }{2\eee6 }{0.90}  & \p{1\eee4 }{5\eee4 }{3.97}  \\ \hline  
    $24$ & \p{2\eee12}{1\eee12}{0.156} & \p{1\eee8 }{5\eee9 }{0.65}  & \p{1\eee6 }{7\eee7 }{2.88}      & \p{5\eee12}{2\eee12}{0.47}  & \p{5\eee8 }{1\eee8 }{1.82}  & \p{5\eee6 }{1\eee6 }{8.06}  \\ \hline  
    $30$ &                             & \p{3\eee11}{2\eee11}{1.05}  & \p{2\eee8 }{7\eee9 }{4.93}      & \p{1\eee12}{2\eee13}{0.60}  & \p{9\eee11}{3\eee11}{2.80}  & \p{3\eee8 }{6\eee9 }{12.6}  \\ \hline  
    $36$ &                             & \p{1\eee12}{8\eee13}{1.53}  & \p{3\eee10}{1\eee10}{8.09}      &                             & \p{3\eee12}{1\eee12}{4.11}  & \p{6\eee10}{1\eee10}{20.8}  \\
\multicolumn{7}{c}{} \\
         & \multicolumn{3}{c|}{$d=2^6$} & \multicolumn{3}{c}{$d=2^7$}  \\
    $r$  & $n=2^5$ & $n=2^7$ & $n=2^9$                                                          & $n=2^5$ & $n=2^7$ & $n=2^9$  \\ \hline
    $ 9$ & \p{6\eee3 }{5\eee4 }{0.15} & \p{2\eee1 }{1\eee1 }{0.70}  &                                 & \p{4\eee3 }{5\eee4 }{0.50} &                             &                           \\ \hline  
    $15$ & \p{2\eee5 }{1\eee6 }{0.45} & \p{3\eee3 }{2\eee4 }{1.85}  & \p{7\eee2 }{3\eee2 }{7.33}      & \p{3\eee6 }{4\eee7 }{1.45} & \p{2\eee3 }{6\eee4 }{5.85}  & \p{1\eee1 }{6\eee2 }{23.6}  \\ \hline  
    $21$ & \p{5\eee10}{1\eee10}{1.01} & \p{3\eee6 }{2\eee7 }{3.70}  & \p{2\eee4 }{4\eee5 }{15.1}      & \p{1\eee9 }{1\eee10}{2.90} & \p{1\eee6 }{2\eee7 }{11.8}  & \p{1\eee3 }{3\eee4 }{48.0}  \\ \hline  
    $27$ & \p{3\eee12}{7\eee13}{1.45} & \p{3\eee9 }{3\eee10}{6.27}  & \p{2\eee6 }{4\eee7 }{26.7}      & \p{1\eee11}{1\eee12}{4.56} & \p{5\eee9 }{6\eee10}{19.9}  & \p{4\eee6 }{5\eee7 }{82.8}  \\ \hline  
    $33$ &                            & \p{1\eee11}{3\eee12}{9.97}  & \p{4\eee9 }{1\eee9 }{42.9}      &                            & \p{3\eee11}{5\eee12}{29.7}  & \p{1\eee8 }{2\eee9 }{128}  \\ \hline  
    $39$ &                            & \p{7\eee12}{2\eee12}{13.1}  & \p{2\eee10}{2\eee11}{64.2}      &                            &                             & \p{3\eee10}{6\eee11}{186}  \\ \hline  
   \end{tabular}
 \end{center}
 \caption{Accuracy and the CPU time for Alg.~\ref{algi} applied for the interpolation of tensor~\eqref{eq:1r} with dimension $d,$  mode size $n$ and TT--ranks $r.$
 Each cell contains the estimates of the relative error in the Chebyshev norm $|A-\tilde A|_{\sim}/|A|_{\sim}$ and  
 in the Frobenius norm $\|A-\tilde A\|_{\sim}/\|A\|_{\sim},$ 
 and the computation time in seconds.
 }
 \label{tab}
\end{table}
We apply Alg.~\ref{algi} to the tensor $A=[A(i_1,\ldots,i_d)]$ with elements
\begin{equation}\label{eq:1r}
 A(i_1,\ldots,i_d) = {1}/{\sqrt{i_1^2 + \ldots + i_d^2}}.
\end{equation}
This example is the standard test considered in e.g.~\cite{ost-tucker-2008,ot-ttcross-2010,lars-htcross-2013}.
We test the algorithm for large mode sizes $n$ and dimensions $d,$ where the evaluation of the accuracy $|A-\tilde A|$ is impossible due to the restrictively large number of entries.
We substitute the exact evaluation by estimates computed on a large number of randomly distributed elements as follows
$$
| A |_{\sim} = \max_{i \in \I} |A(i_1,\ldots,i_d)|, \qquad \|A\|_{\sim}^2 = \frac{n_1\ldots n_d}{\#\I}\sum_{i \in \I} |A(i_1,\ldots,i_d)|^2,
$$
where indices $i=(i_1,\ldots,i_d) \in \I$ are chosen randomly, and  $\#\I$ denotes the number of elements in the random set $\I.$
In our tests $\#\I \geq 2^{30}.$

The results are collected in Tab.~\ref{tab}.
It is not difficult to notice the linear scaling w.r.t. the mode size $n.$
The scaling in dimension is between $\O(d)$ and $\O(d^2),$ since the algorithm requires $\O(d)$ evaluations of tensor elements, and each tensor element depends on $d$ indices.
The scaling in TT--rank is almost quadratic, which shows that the evaluation of tensor elements takes longer than other operations.

For large ranks, the relative accuracy of the interpolation computed by Alg.~\ref{algi} reduces almost to the machine precision threshold and does not stagnate at the level of $10^{-8}$ or $10^{-9},$ cf.~\cite{ot-ttcross-2010,lars-htcross-2013}.
The Alg.~\ref{algi} also appears to be very fast: using one core on the Iridis3 cluster, it is two to three times faster than the HT cross interpolation algorithm~\cite{lars-htcross-2013} applied to the same problem.

\section{Conclusions and future work}
We have generalized two results on the matrix cross interpolation to the tensor case, using the cross interpolation formula~\eqref{eq:ii} proposed by Oseledets and Tyrtyshnikov~\cite{ot-ttcross-2010} for the tensor train format.
First, we have shown that the maximum--volume cross interpolation is quasioptimal, i.e. its accuracy in the Chebyshev norm differs from the best possible accuracy by the factor which does not grow exponentially with dimension.
This generalizes the matrix result of Goreinov and Tyrtyshnikov~\cite{gt-skel-2011}.
Second, we have shown that for the nested interpolation indices formula~\eqref{eq:ii} computes $\sum_{k=1}^d r_{k-1}n_kr_k - \sum_{k=1}^{d-1}r_k^2$ parameters of the TT format inspecting exactly the same number of tensor entries, and on these elements the interpolation is exact.
This generalizes the classical result on the skeleton approximation of matrices to the TT case.

In the tensor case, the maximum--volume interpolation sets in general are not nested, and we cannot have the quasioptimality and the interpolation property simultaneously.
It would be interesting to find the nested interpolation sets which provide a moderate coefficient of the quasioptimality.


Using the interpolation property, we have proposed the fast and simple greedy cross interpolation algorithm, which provides very accurate results for the standard test, and is several times faster than other methods.
Many variants of this algorithm can be developed, taking in account the interpolation property and the available information on the error of the interpolation for different entries of a tensor.
It is easy to overcome the breakdowns, if they occur, simply by taking random pivots in larger subtensors or in the whole tensor, as is suggested in Alg.~\ref{algg}.
In our experiments we have never had a breakdown using the restricted pivoting in Alg.~\ref{algi}.

The theoretical and experimental results of this paper show that the curse of dimensionality cannot stop us from developing fast and reliable cross interpolation methods in higher dimensions.
The cross interpolation allows to convert a given high--dimensional data array into the tensor train format, for which many operations essential for the scientific computing are already possible.
For many high--dimensional problems we can try to substitute the randomized (Monte Carlo) sampling by the cross interpolation in order to benefit from its adaptivity.
This is a subject of further work.

\section*{Acknowledgments}
{\small
The theoretical results of this paper have been obtained when the author was with the Institute of Numerical Mathematics RAS, Moscow.
The author is grateful to Prof.~Eugene Tyrtyshnikov and Dr.~Ivan Oseledets for fruitful discussions.
The author appreciates the use of the Iridis High Performance Computing Facility, and the associated support services at the University of Southampton, that proved essential to carry out the extensive numerical experiments reported in this paper.
The author acknowledges the hospitality of SAM ETH Z\"urich, where the most of the manuscript has been drafted.
}


\end{document}